\documentclass{article}

\usepackage{amsmath,amsthm,amsfonts,amssymb,bm,tikz,xcolor}
\usetikzlibrary{arrows}

\usepackage{diffcoeff}
\usepackage{lineno,hyperref}
\usepackage{soul}
\usepackage{fullpage}
\usepackage{authblk}

\modulolinenumbers[5]

\DeclareMathOperator*{\grad}{ \text{\textbf{grad}}}
\DeclareMathOperator*{\Grad}{Grad}
\DeclareMathOperator*{\Div}{Div}
\DeclareMathOperator*{\Rot}{Rot}
\renewcommand{\div}{\operatorname{div}}

\DeclareMathOperator*{\curl}{ \text{\textbf{curl}}}
\DeclareMathOperator*{\Curl}{Curl}

\newcommand{\bbR}{\mathbb{R}}
\newcommand{\bbC}{\mathbb{C}}

\newcommand*{\norm}[1]{\ensuremath{\left\|#1\right\|}}
\newcommand{\energy}[1]{\frac{1}{2} \int_{\Omega} \left\{ #1 \right\} \d\Omega}

\DeclareMathOperator*{\ran}{ran}

\newtheorem{definition}{Definition}
\newtheorem{theorem}{Theorem}
\newtheorem{proposition}{Proposition}

\newtheorem{remark}{Remark}

\def\onedot{$\mathsurround0pt\ldotp$}
\def\cddot{% two dots stacked vertically
	\mathbin{\vcenter{\baselineskip.67ex
			\hbox{\onedot}\hbox{\onedot}}%
}}

\makeatletter \renewcommand\d[1]{\ensuremath{%
		\;\mathrm{d}#1\@ifnextchar\d{\!}{}}}
\makeatother

\title{Stokes-Dirac structures for distributed parameter port-Hamiltonian systems: an analytical viewpoint}

\author[1]{Andrea Brugnoli}
\author[2]{Ghislain Haine}
\author[2]{Denis Matignon}
\affil[1]{Technische Universit\"at Berlin, Germany}
\affil[2]{ISAE-SUPAERO, Universit\'e{} de Toulouse, France}

\date{}

\begin{document}

\maketitle

\begin{abstract}
In this paper we prove that a large class of linear evolution PDEs defines a Stokes-Dirac structure over Hilbert spaces. To do so, the theory of boundary control system is employed. This definition encompasses problems from mechanics, that cannot be handled by the seminal geometric setting given in [van der Schaft and Maschke, \textit{Hamiltonian formulation of distributed-parameter systems with boundary energy flow, 2002}].
Many worked-out examples stemming from continuum mechanics and physics are presented in detail, and a particular focus is given on the functional spaces in duality at the boundary of the geometrical domain. For each example, the connection between the differential operators and the associated Hilbert complexes is illustrated.


\end{abstract}

\section{Introduction}

The theory of port-Hamiltonian systems (pHs) is an ever-growing area of research~\cite{Beattie2018,Castanos2013,Duindam2009,Rashad2020,VanDerSchaft1998,VanDerSchaft2013,VanDerSchaft2014}, as it provides a powerful framework for modelling~\cite{Altmann2021,Altmann2017,Califano2022,Cardoso2020,Cardoso2017,Gernandt2021,Macchelli2004b,
Maschke1992,Serhani2019a,Vu2016,Zhou2017}, control~\cite{Macchelli2021,Macchelli2004a,Toledo2020} and simulation~\cite{Brugnoli2019a,Brugnoli2019b,Brugnoli2020PHD,Serhani2020PHD} of complex physical systems. Its versatility allows to describe subsystems independently, and to interconnect them through ports~\cite{Cervera2007,Haine2022,Kurula2010,Ortega2002,VanDerSchaft1999}. It models physical exchanges between subsystems, making use of physically meaningful quantities.\\

The geometric characterization of PHs is not univocal in the literature. PH systems can be defined using two approaches:
\begin{itemize}
    \item calculus of variation for field theories and the jet bundle formalism \cite{schoberl2011first,scholberl2014auto},
    \item Dirac structures \cite{courant1990,Kurula2010,Nishida2004,Yoshimura2006,Jimenez2015}.
\end{itemize}
The jet bundle and Dirac structure formalism are likely connected by a unifying geometrical description of pHs. Nevertheless such a connection is yet to be found in the literature.\\

In the jet bundle case, pHs are obtained like in the classical Hamiltonian formalism of symplectic geometric, i.e. by converting the Euler-Lagrange equations via the Legendre transform. The boundary ports are deduced by accounting for non trivial variations on the boundary. This approach is applicable to finite and infinite dimensional systems. In the latter case the Hamiltonian is a functional over a field and its derivatives, commonly named a jet bundle in field theories. This construction has the advantage of being very structured and in direct association with the Lagrangian formalism. However, deducing boundary ports is a non trivial task in higher order problems~\cite{scholberl2018bd}. \\

Dirac structures generalize Hamiltonian systems defined on symplectic manifolds (basic examples of integrable
 Dirac structures are e.g. Poisson and presymplectic manifolds \cite{courant1990}). They describe the energy routing inside and outside of a given system and are strictly connected with graph and network theory \cite{vanderSchaft2013graph}. Solutions of a pH system belong to the Dirac structure at all time. More specifically, this structure encloses the power balance satisfied by the Hamiltonian along trajectories. Twenty years ago, lumped-parameters pHs have been generalized to distributed-parameters pHs \cite{VanDerSchaft2002}, allowing to model in a structured manner the physical exchanges occurring at the boundary of physical domains. This construction allows easily identifying easily the boundary ports of a given distributed system\footnote{The term Stokes-Dirac structure has been coined in \cite{VanDerSchaft2002}, as the boundary variables are provided by the Stokes theorem.}. However, contrarily to the jet bundle description (that relies on the Lagrangian description and its Legendre transform), it is still unclear how to systematically construct pHs in the Dirac structures framework: an overarching geometric definition of pHs based on Dirac structures is yet to be found in the literature and this is especially true in the infinite dimensional case. Many authors have attempted to provide such a unifying definition, starting either from (Stokes-)Dirac structures \cite{LeGorrec2005,Macchelli2004a} or from physically meaningful examples ~\cite{Jacob2012,Skrepek2021,gaybalmaz2017noneqI,gaybalmaz2017noneqII}. For the moment, no general consensus is found in the literature. In particular, the geometric formulation of problems arising from continuum mechanics is, to the best of our knowledge, open.  \\

The present contribution aims at providing a unifying functional analytic framework for linear pHs defined by means of a (Stokes-)Dirac structures. The presented formulation encompasses many engineering examples and relies on the well-known Boundary Control System (BCS) theory, or more generally on well-posed linear systems~\cite{Curtain1989,Kurula2015,Salamon1987,Staffans2005,Tucsnak2009,Tucsnak2014,Weiss2001} to define (Stokes-)\linebreak{}Dirac structures. Such approach has already been used for this purpose~\cite{Jacob2012,LeGorrec2005,Skrepek2021}. The major novelty compared to previous work is that the algebraic structure is clearly separated from the dynamics satisfied by the trajectories. {This is achieved by assuming a particular decomposition of the operators together with an abstract integration by parts formula. An abstract Stokes-Dirac structure is then constructed by means of an auxiliary BCS. To demonstrate well-posedness, only a subclass of linear constitutive relations are considered (namely for undamped linear port-Hamiltonian systems, also called \emph{lossless} pHs).}   Our framework allows to properly describe examples stemming from continuum mechanics, like general elastodynamics and plate models. In the considered physical examples, we highlight the connection between the operator included in the Dirac structure and the associated Hilbert complexes. This connection is important as it establishes a link between algebraic, topological and geometric properties and has important consequences for discretization \cite{arnold2006acta}.\\

The paper is organized as follows: section~\ref{sec:framework} defines the general framework proposed in this work. It is divided in five parts. Section~\ref{sec:finite} recall some useful definitions for lumped-parameters pHs. Then, the definition of Stokes-Dirac structures is given as the direct generalization to infinite-dimensional pHs in Section~\ref{sec:SD}. The definition of a distributed-parameters pHs is detailed in Section~\ref{sec:dpHs}. Section~\ref{sec:BCS2SD} deals with the auxiliary BCS {enabling} the construction of a Stokes-Dirac structure from differential and boundary operators. The class of lossless linear pHs are proved well-posed in Section~\ref{sec:wplpHs}. Section~\ref{sExamples} gives four meaningful physical examples coming from continuum mechanics and physics. Section~\ref{sec:conclusion} concludes this work.

%%%%%%%%%%%%%%%% CAS GENERAL %%%%%%%%%%%%%%%
\section{A general framework}\label{sec:framework}

\subsection{Port-Hamiltonian systems in finite dimension}\label{sec:finite}

A common way to define finite-dimensional port-Hamiltonian systems on $\bbR^N$, borrowed from~\cite{VanDerSchaft2014}, is the port-based modelling, relying on a Dirac structure.

The definition of Dirac structures is given in~\cite[Definition~2.1]{VanDerSchaft2014}, but the equivalent definition given by~\cite[Proposition~2.1]{VanDerSchaft2014} suits better for a generalization to infinite-dimensional spaces.

\begin{definition}[Bond space]
Let $\mathcal{E}$ be a Hilbert space and $\mathcal{F} := \mathcal{E}'$ its topological dual. The space $\mathcal{B} := \mathcal{F} \times \mathcal{E}$ endowed with the bilinear form:
\begin{equation}\label{eq:bond-product}
\left\langle \left\langle
\begin{pmatrix}
f^1 \\ e^1
\end{pmatrix},
\begin{pmatrix}
f^2 \\ e^2
\end{pmatrix}
\right\rangle \right\rangle_{\mathcal{B}}
:=
\left\langle
f^1, e^2
\right\rangle_{\mathcal{F}, \mathcal{E}}
+
\left\langle
f^2, e^1
\right\rangle_{\mathcal{F}, \mathcal{E}}, 
\qquad \forall 
\begin{pmatrix}
f^1 \\ e^1
\end{pmatrix},
\begin{pmatrix}
f^2 \\ e^2
\end{pmatrix} \in \mathcal{B},
\end{equation}
is called a \emph{bond space}. $\mathcal{E}$ is called the \emph{effort space} and $\mathcal{F}$ is called the \emph{flow space}.
\end{definition}

Here $\left\langle
f, e
\right\rangle_{\mathcal{F}, \mathcal{E}} := f(e)$, {\it i.e.} the linear form $f \in (\mathcal{E})'$ applied to the vector $e \in \mathcal{E}$. This notation is classical for general Hilbert spaces, and is known as the duality bracket between $\mathcal{F} := (\mathcal{E})'$ and $\mathcal{E}$. In finite dimension, identification between $(\mathcal{E})'$ and $\mathcal{E}$ is safe since all norms are equivalent. Unfortunately, this is no longer the case in infinite dimension, and continuity of linear maps is norm-dependent, hence a norm has to be chosen and fixed once for all to define the topological dual $(\mathcal{E})'$ of $\mathcal{E}$. 

\begin{definition}[(Stokes-)Dirac structure]
Let $\mathcal{B}$ be a bond space. A subspace $\mathcal{D} \subset \mathcal{B}$ is called a \emph{Dirac or Stokes-Dirac structure} if and only if $\mathcal{D}^{[\perp]} = \mathcal{D}$, where $\mathcal{D}^{[\perp]}$ is the orthogonal companion of $\mathcal{D}$ in $\mathcal{B}$, defined by:
\begin{equation}\label{eq:Stokes-Dirac-Structure}
\mathcal{D}^{[\perp]}
:=
\left\lbrace
\begin{pmatrix}
f^1 \\ e^1
\end{pmatrix} \in \mathcal{B}
\; \mid \;
\left\langle \left\langle
\begin{pmatrix}
f^1 \\ e^1
\end{pmatrix},
\begin{pmatrix}
f^2 \\ e^2
\end{pmatrix}
\right\rangle \right\rangle_{\mathcal{B}} = 0,
\; \forall
\begin{pmatrix}
f^2 \\ e^2
\end{pmatrix} \in \mathcal{D}
\right\rbrace
\end{equation}
\end{definition}

In the real-valued finite-dimensional case, it is common to talk about Dirac structure. In the general framework, we often emphasize the infinite-dimensional setting by talking about Stokes-Dirac structure as it makes use of the so-called Stokes divergence theorem in practice, see {\it e.g.}~\cite{Rashad2020}.

\begin{definition}[Port-Hamiltonian systems~\cite{VanDerSchaft2014}]\label{def:dimFinPHS}
Consider a \emph{solution space}\footnote{Note that this is not the usual {\em state space}, which is usually determined by the ``energy'' norm, of the solutions, but a strict subspace. This terminology follows {\it e.g.}~\cite{Tucsnak2009}.} $\mathcal{Z}$, a resistive space $\mathcal{R}$, a \emph{control space} $\mathcal{U}$, and $\mathcal{H} : \mathcal{Z} \rightarrow \bbR$ a \emph{Hamiltonian} defining energy-storage, function of the energy variable $\alpha$. A port-Hamiltonian system on $\left(\mathcal{Z}, \mathcal{R}, \mathcal{U}\right) \simeq \left( \bbR^{d_s} \times \bbR^{d_r} \times \bbR^{d_u} \right)$ is defined by a Dirac structure:
$$
\mathcal{D} \subset ( \mathcal{Z}' \times \mathcal{R}' \times \mathcal{U}' ) \times ( \mathcal{Z} \times \mathcal{R} \times \mathcal{U} ),
$$
and a dynamics ({\it i.e.} \emph{trajectories} depending on the initial value $\alpha(0)$ and on the control $u(t)$) evolving in this Dirac structure:
$$
\left( \begin{pmatrix}
\dot{\alpha}(t) \\ f_r(t) \\ -y(t)
\end{pmatrix} , \begin{pmatrix}
\grad_{\alpha} \mathcal{H}(\alpha(t)) \\ e_r(t) \\ u(t)
\end{pmatrix} \right) \in \mathcal{C}([0,\infty); \mathcal{D}),
$$
together with a \emph{constitutive relation} for the resistive port $(f_r, e_r) \in \mathcal{S} \subset \mathcal{R}' \times \mathcal{R}$.
\end{definition}

In~\cite[Definition~2.3]{VanDerSchaft2014}, the Dirac structure depends on the energy variables~$\alpha$: the Dirac structure is \emph{modulated}. Modulated Stokes-Dirac structures for distributed port-Hamiltonian system are discussed in  {\it e.g.}~\cite{Califano2022,Cardoso2020}. In the present work, only \emph{constant} Stokes-Dirac structures are considered. 

\begin{proposition}[Power-balance]\label{prop:power-balance-dim-finie}
The Hamiltonian of a port-Hamiltonian system satisfies the following \emph{power-balance} along the trajectories:
\begin{equation}\label{eq:dimFinPB}
\frac{{\rm d}}{{\rm d} t} \mathcal{H}(\alpha(t)) = -\left( f_r, e_r(t) \right)_{\bbR^{d_r}} + \left(y(t), u(t)\right)_{\bbR^{d_u}}, \qquad \forall t\ge0.
\end{equation}
In particular, $\mathcal{H} \in \mathcal{C}^1([0,\infty); \bbR)$.
\end{proposition}

In practice, the constitutive relation $\mathcal{S}$ is given and the term $\left( f_r, e_r(t) \right)_{\bbR^{d_r}}$ leads to a fully determined power-balance~\eqref{eq:dimFinPB}.
\begin{proposition}[Extended structure matrix]\label{prop:extended-finite-dim}
Consider a port-Hamiltonian system and assume that the trajectories are solutions of the following system:
\begin{equation}\label{eq:dimFinPHSstateSpace}
\begin{pmatrix} \dot \alpha(t) \\ f_r(t) \end{pmatrix} = J \begin{pmatrix} \grad_{\alpha} \mathcal{H}(\alpha(t)) \\ e_r(t) \end{pmatrix} + B u(t), \quad y(t) = B^\top \begin{pmatrix} \grad_{\alpha} \mathcal{H}(\alpha(t)) \\ e_r(t) \end{pmatrix},
\end{equation}
where $J$ is a skew-symmetric matrix and $B$ a control matrix, with appropriate sizes.

Then the Dirac structure is given as the graph\footnote{Rigorously speaking, the Dirac structure is given by the graph of the \textit{inverse} $\mathcal{J}^{-1}$ of the extended structure matrix since it is regarded in $\mathcal{F} \times \mathcal{E}$ and not in $\mathcal{E} \times \mathcal{F}$. Nevertheless, to avoid mentioning details on the \emph{inverse}, by language abuse, we will only say {\em graph} throughout this paper, since it perfectly fits the definition of~\cite[Exercise~1, p~17]{VanDerSchaft2014}.} of the \emph{extended structure matrix}:
$$
\mathcal{J} := \begin{bmatrix}
J & B \\ -B^\top & 0
\end{bmatrix}.
$$
\end{proposition}

\begin{proof}
\cite[Exercise~1, p~17]{VanDerSchaft2014}.
\end{proof}

\subsection{Stokes-Dirac structure over complex Hilbert spaces}\label{sec:SD}

This section aims at providing a sufficient condition for an operator on complex Hilbert spaces to generate a Stokes-Dirac structure as its graph.
The definition of Bond space extends from finite to infinite dimension. However, for an infinite-dimensional system, it is compulsory to consider complex spaces and to consider a weaker topology on the flow space $\mathcal{F} := \mathcal{E}'$, as will be seen in the sequel. This dictates the following \emph{complex}\footnote{The topological dual $\mathcal{F}$ of $\mathcal{E}$ being the vector space of continuous \emph{linear} forms on $\mathcal{E}$, $\langle f, e \rangle_{\mathcal{F},\mathcal{E}} := f(e)$ is antilinear in its first variable, and linear in its second.} definition of a \emph{sesquilinear form} as bond product on $\mathcal{B} := \mathcal{E}' \times \mathcal{E}$:
$$
\left\langle \left\langle
\begin{pmatrix}
f^1 \\ e^1
\end{pmatrix},
\begin{pmatrix}
f^2 \\ e^2
\end{pmatrix}
\right\rangle \right\rangle_{\mathcal{B}}
:=
\overline{
\left\langle
f^1, e^2
\right\rangle_{\mathcal{F}, \mathcal{E}}
}
+
\left\langle
f^2, e^1
\right\rangle_{\mathcal{F}, \mathcal{E}}, 
\qquad \forall 
\begin{pmatrix}
f^1 \\ e^1
\end{pmatrix},
\begin{pmatrix}
f^2 \\ e^2
\end{pmatrix} \in \mathcal{B},
$$

\begin{theorem}\label{th:Stokes-Dirac}
Let $\mathcal{E}$ be a Hilbert space, $\mathcal{F} = \mathcal{E}'$ its topological dual, and $\mathcal{J} \in \mathcal{L}(\mathcal{E}, \mathcal{F})$. If~\footnote{This identity is a skew-symmetry\emph{-like} property of $\mathcal{J}$. The classical skew-symmetry would {require} that $\mathcal{J}$ has its range in $\mathcal{E}$, and to make use of the Hermitian product of $\mathcal{E}$ instead of the duality bracket between $\mathcal{F}$ and $\mathcal{E}$.}:
\begin{equation}\label{eq:skew-sym}
\left\langle
\mathcal{J} e^1, e^2
\right\rangle_{\mathcal{F}, \mathcal{E}}
= -
\overline{
\left\langle
\mathcal{J} e^2, e^1
\right\rangle_{\mathcal{F}, \mathcal{E}}},
\qquad \forall e^1, e^2 \in \mathcal{E},
\end{equation}
then:
$$
\mathcal{D} := {\rm Graph} (\mathcal{J}) 
:= 
\left\lbrace 
\begin{pmatrix}
\mathcal{J} e \\ e
\end{pmatrix} \in \mathcal{B}
\; \mid \;
\forall e \in \mathcal{E}
\right\rbrace,
$$
is a Stokes-Dirac structure in $\mathcal{B} := \mathcal{F} \times \mathcal{E}$.

The operator $\mathcal{J}$ is called the \emph{extended structure operator} of $\mathcal{D}$.

\iffalse
=== OLD ===

Let $\mathcal{E}$ be a Hilbert space, $\mathcal{F}$ its topological dual, and $\mathcal{J}$ a closed and densely defined linear operator from $\mathcal E$ to $\mathcal{F}$. In particular, the domain of $\mathcal{J}$, defined as:
$$
\mathcal{E}_1 := D(\mathcal{J}) := \left\{ e \in \mathcal{E} \, \mid \, \mathcal{J} e \in \mathcal{F} \right\},
$$
is a Hilbert space endowed with the graph norm, continuously and densely embedded in $\mathcal{E}$, and $\mathcal{J} \in \mathcal{L}(\mathcal{E}_1, \mathcal{F})$. If~\footnote{This identity is a skew-symmetry\emph{-like} property of $\mathcal{J}$. The classical skew-symmetry would {require} that $\mathcal{J}$ has its range in $\mathcal{E}$, and to make use of the scalar product of $\mathcal{E}$ instead of the duality bracket between $\mathcal{F}$ and $\mathcal{E}$.}:
\begin{equation}\label{eq:skew-sym}
\left\langle
\mathcal{J} e^1, e^2
\right\rangle_{\mathcal{F}, \mathcal{E}}
= -
\left\langle
\mathcal{J} e^2, e^1
\right\rangle_{\mathcal{F}, \mathcal{E}},
\; \forall e^1, e^2 \in \mathcal{E}_1,
\end{equation}
then:
$$
\mathcal{D} := {\rm Graph} (\mathcal{J}) 
:= 
\left\lbrace 
\begin{pmatrix}
\mathcal{J} e \\ e
\end{pmatrix} \in \mathcal{B}
\; \mid \;
\forall e \in \mathcal{E}_1
\right\rbrace,
$$
is a Stokes-Dirac structure in $\mathcal{B} := \mathcal{F} \times \mathcal{E}$.

The operator $\mathcal{J}$ is called the \emph{extended structure operator} of $\mathcal{D}$.

=== OLD ===
\fi
\end{theorem}

\begin{proof}
Let $\begin{pmatrix}
f^1 \\ e^1
\end{pmatrix} \in \mathcal{D}$.
Then for all $\begin{pmatrix}
f^2 \\ e^2
\end{pmatrix} \in \mathcal{D}$, one has:
\begin{equation*}
\begin{array}{rcl}
\displaystyle \left\langle \left\langle \begin{pmatrix} f^1 \\ e^1 \end{pmatrix}, \begin{pmatrix} f^2 \\ e^2 \end{pmatrix} \right\rangle \right\rangle_{\mathcal{B}} 
&=& \displaystyle \overline{\left\langle f^1, e^2 \right\rangle_{\mathcal{F}, \mathcal{E}}} + \left\langle f^2, e^1 \right\rangle_{\mathcal{F}, \mathcal{E}}, \\
&=& \displaystyle \overline{\left\langle \mathcal{J} e^1, e^2 \right\rangle_{\mathcal{F}, \mathcal{E}}} + \left\langle \mathcal{J} e^2, e^1 \right\rangle_{\mathcal{F}, \mathcal{E}}, \\
&\overset{\eqref{eq:skew-sym}}{=}& \displaystyle \overline{\left\langle \mathcal{J} e^1, e^2 \right\rangle_{\mathcal{F}, \mathcal{E}}} - \overline{\left\langle \mathcal{J} e^1, e^2 \right\rangle_{\mathcal{F}, \mathcal{E}}}, \\
&=& \displaystyle 0.
\end{array}
\end{equation*}
This shows that $\mathcal{D} \subset \mathcal{D}^{[\perp]}$, {\it i.e.} that $\mathcal{D}$ is a \emph{Tellegen structure}.

Reciprocally, let $\begin{pmatrix} f^1 \\ e^1 \end{pmatrix} \in \mathcal{D}^{[\perp]}$. Then for all $\begin{pmatrix} f^2 \\ e^2 \end{pmatrix} \in \mathcal{D}$, one has:
\begin{equation*}
\begin{array}{rcl}
\displaystyle 0 
&=& \displaystyle \left\langle \left\langle \begin{pmatrix} f^1 \\ e^1 \end{pmatrix}, \begin{pmatrix} f^2 \\ e^2 \end{pmatrix} \right\rangle \right\rangle_{\mathcal{B}} \\
&=& \displaystyle \overline{\left\langle f^1, e^2 \right\rangle_{\mathcal{F}, \mathcal{E}}} + \left\langle \mathcal{J} e^2, e^1 \right\rangle_{\mathcal{F}, \mathcal{E}}, \\
&=& \displaystyle \overline{\left\langle f^1, e^2 \right\rangle_{\mathcal{F}, \mathcal{E}}} - \overline{\left\langle \mathcal{J} e^1, e^2 \right\rangle_{\mathcal{F}, \mathcal{E}}}.
\end{array}
\end{equation*}
This is true for all $e_2 \in \mathcal{E}$, hence $f^1 - \mathcal{J} e_1 \in \mathcal{E}^\perp := \left\{ f \in \mathcal{F} \, \mid \, \left\langle f, e \right\rangle_{\mathcal{F}, \mathcal{E}} = 0 \text{ for all } e \in \mathcal{E} \right\} \equiv \{ 0 \}$, {\it i.e.} $\begin{pmatrix} f^1 \\ e^1 \end{pmatrix} \in \mathcal{D}$, which concludes the proof.

\end{proof}

\begin{remark}
Let us consider the following example to fix the ideas: $\mathcal{J} := \frac{\rm d}{\rm d x}$ defined from the Sobolev space $H^1_0(0,1)$ to $L^2(0,1)$. There is two ways for considering this operator: either as a closed and densely-defined {\em unbounded} operator from $L^2(0,1)$ to $L^2(0,1)$, or as a {\em bounded} operator from $H^1_0(0,1)$ to $L^2(0,1)$. If one takes $\mathcal{E} := L^2(0,1) = \mathcal{F}$, the graph of $\mathcal{J}$ would only be a Tellegen structure on $\mathcal{F} \times \mathcal{E} = L^2(0,1) \times L^2(0,1)$. Indeed, the reciprocal part of the above proof would require $\mathcal{J}$ to be {\em skew-adjoint} to hold, which would be too restrictive for our purpose. On the other hand, if one takes $\mathcal{E} := H^1_0(0,1)$ and $\mathcal{F} := (H^1_0(0,1))' \simeq H^{-1}(0,1) \supset L^2(0,1)$, one obtains a Dirac structure on $\mathcal{F} \times \mathcal{E} = H^{-1}(0,1) \times H^1_0(0,1)$ as expected. The price to pay is the weaker topology on the flow space $\mathcal{F}$, and the stronger one on the effort space $\mathcal{E}$.
\end{remark}

\begin{remark}
Equivalently, the \emph{skew-symmetric-like} property~\eqref{eq:skew-sym} can be rewritten as follows:
$$
\Re {\rm e} \left\langle
\mathcal{J} e, e
\right\rangle_{\mathcal{F}, \mathcal{E}}
= 0,
\qquad \forall e \in \mathcal{E}.
$$
\end{remark}

\begin{remark}
Theorem~\ref{th:Stokes-Dirac} gives a \emph{kernel representation} of the infinite-dimensional Stokes-Dirac structure $\mathcal{D}$, as defined in~\cite[Chapter~5]{VanDerSchaft2014} for finite-dimensional port-Hamiltonian systems.
\end{remark}

\subsection{Port-Hamiltonian systems on Hilbert spaces}\label{sec:dpHs}

Assuming $\mathcal{Z}$, $\mathcal{R}$ and $\mathcal{U}$ to be Hilbert spaces, Definition~\ref{def:dimFinPHS} directly translates to the infinite-dimensional setting, providing the gradient of the Hamiltonian $\grad_{\alpha} \mathcal{H}$ is replaced by the variational derivate $\delta_{\alpha} \mathcal{H}$, whose definition can be directly extended to our framework following {\it e.g.}~\cite[Definition~4.1,~p.~245]{Olver93}. In particular, Proposition~\ref{prop:power-balance-dim-finie} admits a straightforward generalization.
\begin{proposition}[Power-balance]\label{prop:power-balance-dim-inf}
Let $\mathcal{Z}$, $\mathcal{R}$, and $\mathcal{U}$ be three Hilbert spaces, and a functional $\mathcal{H} : \mathcal{Z} \rightarrow \bbR$ a Hamiltonian, function of the energy variable $\alpha$, defining energy-storage.

Consider a port-Hamiltonian system on $\left(\mathcal{Z},\mathcal{R},\mathcal{U}\right)$ defined by a Stokes-Dirac structure:
$$
\mathcal{D} \subset ( \mathcal{Z}' \times \mathcal{R}' \times \mathcal{U}' ) \times ( \mathcal{Z} \times \mathcal{R} \times \mathcal{U} ),
$$
and trajectories (depending on the initial value $\alpha(0)$ and on the control $u$):
$$
\left( \begin{pmatrix}
\dot{\alpha}(t) \\ f_r(t) \\ -y(t)
\end{pmatrix} , \begin{pmatrix}
\delta_{\alpha} \mathcal{H}(\alpha(t)) \\ e_r(t) \\ u(t)
\end{pmatrix} \right) \in \mathcal{C}([0,\infty); \mathcal{D}),
$$
together with a \emph{{resistive} constitutive relation} for the resistive port $(f_r, e_r) \in \mathcal{S} \subset \mathcal{R}' \times \mathcal{R}$.

Then the Hamiltonian $\mathcal{H}(\alpha(t)) \in \mathcal{C}^1([0,\infty);\bbR)$ satisfies the following \emph{power-balance} along the trajectories:
\begin{equation}\label{eq:dimInfPB}
\frac{{\rm d}}{{\rm d} t} \mathcal{H}(\alpha(t)) = - \Re {\rm e}  \left\langle f_r(t), e_r(t) \right\rangle_{\mathcal{R}',\mathcal{R}} + \Re {\rm e} \left\langle y(t), u(t) \right\rangle_{\mathcal{U}',\mathcal{U}}, \qquad \forall t\ge0.
\end{equation}
\end{proposition}

\begin{remark}
As in the finite-dimensional setting~\eqref{eq:dimFinPB}, the resistive constitutive relation is needed to relate $f_r$ and $e_r$, and conclude to the final power-balance. This supplementary constitutive relation often models a \emph{dissipation} through a proportional law (such that Ohm's law, Fourier's law, etc.). Indeed, assume that $\mathcal{R}' \simeq \mathcal{R}$ and there exists $S \in \mathcal{L}(\mathcal{R}, \mathcal{R})$, such that $e_r = S f_r$, where $S$ is symmetric and positive: $S^\star=S$ and $\left( f_r, S f_r \right)_{\mathcal{R}} \ge 0$ for all $f_r \in \mathcal{R}$, then $\Re {\rm e} \left\langle f_r(t), e_r(t) \right\rangle_{\mathcal{R}', \mathcal{R}} = \left( f_r(t), S f_r(t) \right)_{\mathcal{R}} \ge 0$ and the power-balance then reads:
$$
\frac{\rm d}{{\rm d} t} \mathcal{H}(\alpha(t))
= - \left( f_r(t), S f_r(t) \right)_{\mathcal{R}} 
+ \Re {\rm e} \left\langle y(t), u(t) \right\rangle_{\mathcal{U}', \mathcal{U}}
\le \Re {\rm e} \left\langle y(t), u(t) \right\rangle_{\mathcal{U}', \mathcal{U}}, \qquad \forall t \ge 0,
$$
which stands for \emph{lossy} port-Hamiltonian systems.
\end{remark}

\subsection{Formal skew-symmetry with boundary control and structure operator}\label{sec:BCS2SD}

This section is devoted to the description of a class of operators generating a Stokes-Dirac structure as its graph thanks to Theorem~\ref{th:Stokes-Dirac}. The aim is to obtain an infinite-dimensional counterpart of Proposition~\ref{prop:extended-finite-dim}, namely Theorem~\ref{th:BCS2SD}.

As starting point, Boundary Control Systems (BCS) are considered as infinite-dimensional analogous of systems of the form~\eqref{eq:dimFinPHSstateSpace}.

More precisely, we consider system of the form: 
\begin{equation}\label{eq:dimInfPHSstateSpace}
\begin{pmatrix} \dot \alpha(t) \\ f_r(t) \end{pmatrix} = J \begin{pmatrix} \delta_{\alpha} \mathcal{H}(\alpha(t)) \\ e_r(t) \end{pmatrix},
\qquad
G \, \delta_{\alpha} \mathcal{H}(t) = u(t),
\end{equation}

\iffalse
{This has to be rewritten as it is not this kind of systems. The control appear only at the level of the dynamical variables}
\begin{equation}\label{eq:dimInfPHSstateSpace}
\begin{pmatrix} \dot \alpha(t) \\ f_r(t) \end{pmatrix} = J \begin{pmatrix} \delta_{\alpha} \mathcal{H}(t) \\ e_r(t) \end{pmatrix},
\qquad
\left\lbrace
\begin{array}{rcl}
\gamma^1 \, \delta_{\alpha} \mathcal{H}(t) &=& u^1(t), \\
\gamma^2 \, e_r(t) &=& u^2(t),
\end{array}\right.
\end{equation}

{ I propose this correction
\begin{equation}\label{eq:dimInfPHSstateSpace}
\begin{pmatrix} 
\dot \alpha_1(t) \\ 
\dot \alpha_2(t) \\ 
f_r(t) \end{pmatrix} = 
J 
\begin{pmatrix} 
\delta_{\alpha_1} \mathcal{H}(t) \\ 
\delta_{\alpha_2} \mathcal{H}(t) \\ 
e_r(t) 
\end{pmatrix},
\qquad
\left\lbrace
\begin{array}{rcl}
\gamma^1 \, \delta_{\alpha_1} \mathcal{H}(t) &=& u^1(t), \\
\gamma^2 \, \delta_{\alpha_2} \mathcal{H}(t) &=& u^2(t),
\end{array}\right.
\end{equation}
}
\fi
where $J$ is \emph{formally} skew-symmetric and $G$ a boundary control operator. The output will be defined accordingly to $G$ in order to obtain the desired Stokes-Dirac structure.

%$\gamma^i$ are boundary control operators, $i=1,2$. The output will be defined accordingly to $\gamma^i$, $i=1,2$, in order to obtain the desired Stokes-Dirac structure.

An important point to keep in mind for this section is that the focus is {set} on the Stokes-Dirac structure, and not on the trajectories of a port-Hamiltonian system. Therefore there is no need to distinguish the solution space $\mathcal{Z}$ from the resistive space $\mathcal{R}$, neither to consider the time evolution.

Furthermore, many examples coming from physics, as will be seen in Section~\ref{sExamples}, lead us to introduce some notations, and  propose the following assumptions:
\begin{itemize}
\item[\textbf{(A1)}]
$\mathcal{X}^1$ and $\mathcal{X}^2$ are two Hilbert spaces, identified with their respective duals;
\item[\textbf{(A2)}]
$J$ can be decomposed as $\begin{bmatrix} 0 & - K \\ L & 0 \end{bmatrix}$ on $\mathcal{X} := \mathcal{X}^1 \times \mathcal{X}^2$;
\item[\textbf{(A3)}]
$L$ is a closed and densely defined operator from $\mathcal{X}^1$ into $\mathcal{X}^2$, with domain $\mathcal{Z}^1$. Endowed with the graph norm, $\mathcal{Z}^1$ is a Hilbert space, continuously and densely embedded in $\mathcal{X}^1$ \cite[Section~2.2]{Tucsnak2009}, and $L \in \mathcal{L}(\mathcal{Z}^2,\mathcal{X}^1)$; and $K$ is a closed and densely defined operator from $\mathcal{X}^2$ into $\mathcal{X}^1$, with domain $\mathcal{Z}^2$. The Hilbert space $\mathcal{Z}^2$ is also endowed with the graph norm, and $K \in \mathcal{L}(\mathcal{Z}^1,\mathcal{X}^2)$. Furthermore, their domains satisfy $\mathcal{Z}^1 \times \mathcal{Z}^2 = \mathcal{Z} \times \mathcal{R}$;
\item[\textbf{(A4)}]
$G$ can be decomposed as $\begin{bmatrix} \gamma^1 & 0 \\ 0 & \gamma^2 \end{bmatrix} \in \mathcal{L}(\mathcal{Z}^1\times\mathcal{Z}^2, \mathcal{U}^1\times\mathcal{U}^2)$, with $\mathcal{U}^1$ and $\mathcal{U}^2$ two other Hilbert spaces satisfying $\mathcal{U} = \mathcal{U}^1 \times \mathcal{U}^2$.
\end{itemize}

Figure~\ref{fig:decomposition-J} shows the interactions between the different spaces of our setting.
\begin{figure}
\centering
\begin{tikzpicture}[scale=0.7]
\draw (-6,2) node{$\mathcal{U}^2$};
\draw[->>, very thick] (-2.5,2) -- (-5.5,2);
\draw (-4,2) node[below]{$\gamma^2$};
\draw (-2,2) node{$\mathcal{Z}^2$};
\draw[-latex, very thick] (-1.5,2) -- (1.5,2);
\draw (0,2) node[below]{$K$};
\draw (2,2) node{$\mathcal{X}^1$};

\draw (-2,-2) node{$\mathcal{X}^2$};
\draw[-latex, very thick] (1.5,-2) -- (-1.5,-2);
\draw (0,-2) node[above]{$L$};
\draw (2,-2) node{$\mathcal{Z}^1$};
\draw[->>, very thick] (2.5,-2) -- (5.5,-2);
\draw (4,-2) node[above]{$\gamma^1$};
\draw (6,-2) node{$\mathcal{U}^1$};

\draw[left hook-latex, dashed, very thick] (-2,1) -- (-2,-1);
\draw[left hook-latex, dashed, very thick] (2,-1) -- (2,1);
\end{tikzpicture}
\caption{Relations between the spaces and the continuous linear operators. Each arrow represents an operator, a doubled-headed arrow means that it is surjective, and the hooked and dashed arrows mean dense injections.}\label{fig:decomposition-J}
\end{figure}

Finally, it has been assumed in system~\eqref{eq:dimInfPHSstateSpace} that $J$ is \emph{formally} skew-symmetric, which translates with the above decomposition of $J$ by $L$ and $K$ being \emph{formal adjoints with respect to $G$}:
$$
\left( L e^1, e^2 \right)_{\mathcal{X}^2} = \left( e^1, K e^2 \right)_{\mathcal{X}^1},
\qquad \forall \begin{pmatrix} e^1 \\ e^2 \end{pmatrix} \in \ker G.
$$
This identity can be seen as an abstract formulation of the usual definition of \emph{formal adjoints}, often encountered in the port-Hamiltonian formalism for differential operators, and using $\mathcal{C}^\infty_0$ test functions, see \cite[Def. 5.80]{renardy2006introduction}.

{In the framework of this paper, a slightly more general assumption is being made:
\begin{itemize}
\item[\textbf{(A5)}] There exists two operators $\beta^1 \in \mathcal{L}(\mathcal{Z}^1,(\mathcal{U}^2)')$ and $\beta^2 \in \mathcal{L}(\mathcal{Z}^2,(\mathcal{U}^1)')$ such that the following abstract Green's identity holds:
\end{itemize}
\begin{equation}\label{eq:formal-adjoint}
\left( L e^1, e^2 \right)_{\mathcal{X}^2} = \left( e^1, K e^2 \right)_{\mathcal{X}^1}
{+ \left\langle \gamma^1 e^1, \beta^2 e^2 \right\rangle_{\mathcal{U}^1,(\mathcal{U}^1)'}
+ \left\langle \beta^1 e^1, \gamma^2 e^2 \right\rangle_{(\mathcal{U}^2)',\mathcal{U}^2}},
\quad \forall \begin{pmatrix} e^1 \\ e^2 \end{pmatrix} \in {\mathcal{Z}^1 \times \mathcal{Z}^2}.
\end{equation}}
%\end{itemize}}
\noindent
{It is clear that this abstract Green's identity implies that $L$ and $K$ are formal adjoints with respect to $G$.}

{\begin{remark}
    The abstract integration by parts formula  \eqref{eq:formal-adjoint} has an important connection with differential geometry.
    When dealing with the de Rham complex, it corresponds to the topological integration by parts formula of differential forms \cite[Eq. 2.4]{arnold2006acta}. When the elasticity complex is considered the corresponding formula is based on the exterior covariant derivative (see for instance \cite[Eq. 32]{rashad2023intrinsic}).
\end{remark}}
To retrieve the notations of the previous sections, we may {\it e.g.} consider either $\mathcal{Z} = \mathcal{Z}^1 \times \mathcal{Z}^2$ and $\mathcal{R} = \emptyset$, or $\mathcal{Z} = \mathcal{Z}^1$ and $\mathcal{R} = \mathcal{Z}^2$. The former will be our setting for well-posedness (see Theorem~\ref{th:well-posed}) and most examples treated in Section~\ref{sExamples}. An example with $\mathcal{R}$ being neither $\emptyset$ nor $\mathcal{Z}^2$ is provided in Section~\ref{s-curl3D}.

{
\begin{remark}
In this work, an abstract Green's identity~\eqref{eq:formal-adjoint} is assumed from the very beginning, contrarily to the point of view, although equivalent, developed in {\it e.g.}~\cite[Def~2.1]{Kurula2015},~\cite[Def~2.1]{Skrepek2021}, or \cite[Def~4.1]{Wegner2017}. As an example, if $L=\div$ on $H^{\div}$ and $K=-\grad$ on $H^1$, we assume:
$$
\int_\Omega f \div \bm{g} = - \int_\Omega \grad f \cdot \bm{g} + \left\langle \gamma_0 f, \gamma_{\bm{n}} \bm{g} \right\rangle_{H^\frac{1}{2}, H^{-\frac{1}{2}}},
$$
rather than:
$$
\int_\Omega \begin{pmatrix} \bm{f} \\ f \end{pmatrix} \cdot \begin{pmatrix} \grad g \\ \div \bm{g} \end{pmatrix}
=
- \int_\Omega \begin{pmatrix} \grad f \\ \div \bm{f} \end{pmatrix} \cdot \begin{pmatrix} \bm{g} \\ g \end{pmatrix}
+
\left\langle \gamma_0 f, \gamma_{\bm{n}} \bm{g} \right\rangle_{H^\frac{1}{2}, H^{-\frac{1}{2}}}
+
\left\langle \gamma_0 g, \gamma_{\bm{n}} \bm{f} \right\rangle_{H^\frac{1}{2}, H^{-\frac{1}{2}}},
$$
for the computation of the scalar product   $\quad \left( \begin{pmatrix} \bm{f} \\ f \end{pmatrix} , 
J\,\begin{pmatrix} \bm{g} \\ g \end{pmatrix} 
\right)_{L^2(\Omega;\bbR^3) \times L^2(\Omega)}$ 
where  $J := \begin{bmatrix} 0 & \grad \\ \div & 0 \end{bmatrix}$. 
% on $L^2(\Omega;\bbR^3) \times L^2(\Omega)$.
\end{remark}
}
\iffalse
{Here should be added a remark on the link with the (very interesting) reference Wegner (2017), and especially the identity (BT-1) involving $A^\star$. Our framework takes place between the one in Wegner (2017) (pure functional analysis) and physically sounded PDE. More precisaly, when talking about Green's formula, we have in mind:
$$
\int_\Omega f \div \bm{g} = - \int_\Omega \grad f \cdot \bm{g} + \left\langle \gamma_0 f, \gamma_{\bm{n}} \bm{g} \right\rangle_{H^\frac{1}{2}, H^{-\frac{1}{2}}},
$$
rather than:
$$
\int_\Omega \begin{pmatrix} f \\ \bm{f} \end{pmatrix} \cdot \begin{pmatrix} \div \bm{g} \\ \grad g \end{pmatrix}
=
- \int_\Omega \begin{pmatrix} \div \bm{f} \\ \grad f \end{pmatrix} \cdot \begin{pmatrix} g \\ \bm{g} \end{pmatrix}
+
\left\langle \gamma_0 f, \gamma_{\bm{n}} \bm{g} \right\rangle_{H^\frac{1}{2}, H^{-\frac{1}{2}}}
+
\left\langle \gamma_0 g, \gamma_{\bm{n}} \bm{f} \right\rangle_{H^\frac{1}{2}, H^{-\frac{1}{2}}}
$$
In other words, we are more focused on \emph{usual} Green's identities and the relation between the operators $K$ and $L$ to generate a Stokes-Dirac structure than on the \emph{boundary triplet} structure itself and the characterization of the boundary operators allowed for a given $A^\star$.
}
\fi

Before going further in the port-Hamiltonian framework, let us show that the above assumptions allow the definition of the skew-adjoint operator which will be the $C_0$-semi-group generator of the boundary control system generating the Stokes-Dirac structure.
\begin{theorem}\label{th:skew-adjoint}
Assume that $\mathcal{X}^i_1 := \ker \gamma^i$ is dense in $\mathcal{X}^i$, for $i=1, 2$, and denote $\mathcal{X}_1 := \mathcal{X}^1_1 \times \mathcal{X}^2_1$. Let us define the operator $A$ as the restriction of $J$ to $\mathcal{X}_1$, {\it i.e.} $A := J|_{\mathcal{X}_1}$. If $\mathcal{X}^1_0 := \ker \gamma^1 \cap \ker \beta^1$ is dense in $\mathcal{X}^1$ and $\beta^1$ restricted to $\mathcal{X}^1_1$ is onto, then $A$ is skew-adjoint on $\mathcal{X}$.
\end{theorem}

{\begin{remark}
Regarding the symmetric role played by $L$ and $K$, the latter hypothesis may be replaced by: $\mathcal{X}^2_0 := \ker \gamma^2 \cap \ker \beta^2$ is dense in $\mathcal{X}^2$ and $\beta^2$ restricted to $\mathcal{X}^2_1$ is onto, taking care of the obvious reversal needed in the following proof. Furthermore, if both $\mathcal{X}^1_0$ and $\mathcal{X}^2_0$ are dense in $\mathcal{X}^1$ and $\mathcal{X}^2$ respectively, then both $\mathcal{X}^1_1$ and $\mathcal{X}^2_1$ are dense in $\mathcal{X}^1$ and $\mathcal{X}^2$ respectively. {If $L$ and $K$ belong to the de Rham complex, this symmetrical construction finds its explanation in the Hodge duality, that converts a strong differential operator, i.e. the exterior derivative, into a weak one, i.e the codifferential. A more involved notion of Hodge duality arises if one considers the elasticity complex \cite[Section 5]{rashad2023intrinsic}.}
\end{remark}}

\begin{proof} {By hypothesis, $A$ is a densely-defined on $\mathcal{X}$. Furthermore, \eqref{eq:formal-adjoint} implies that $A$ is skew-symmetric. Indeed, one has for all $z = \begin{pmatrix} z_1 \\ z^2 \end{pmatrix} \in \mathcal{X}_1 := \mathcal{X}^1_1 \times \mathcal{X}^2_1 := \ker \gamma^1 \times \ker \gamma^2$:
$$
\begin{array}{rcl}
\Re{e} \left( A z, z \right)_{\mathcal{X}} &=& \Re{e} \left( \begin{pmatrix} -K z^2 \\ L z^1 \end{pmatrix}, \begin{pmatrix} z_1 \\ z^2 \end{pmatrix} \right)_{\mathcal{X}} \\
&=& \Re{e} \left( \left( L z^1, z^2 \right)_{\mathcal{X}^2} - \left( K z^2, z^1 \right)_{\mathcal{X}^1} \right) \\
&=& \Re{e} \left( \left( z^1, K z^2 \right)_{\mathcal{X}^1} - \left( K z^2, z^1 \right)_{\mathcal{X}^1}
+ \left\langle \gamma^1 z^1, \beta^2 z^2 \right\rangle_{\mathcal{U}^1,(\mathcal{U}^1)'}
+ \left\langle \beta^1 z^1, \gamma^2 z^2 \right\rangle_{(\mathcal{U}^2)',\mathcal{U}^2} \right) \\
&=& \Re{e} \left( \left\langle \gamma^1 z^1, \beta^2 z^2 \right\rangle_{\mathcal{U}^1,(\mathcal{U}^1)'}
+ \left\langle \beta^1 z^1, \gamma^2 z^2 \right\rangle_{(\mathcal{U}^2)',\mathcal{U}^2} \right) \\
&=& 0,
\end{array}
$$
since $z^1 \in \ker \gamma^1$ and $z^2 \in \ker \gamma^2$.}

{The aim is to apply \cite[Proposition~3.7.3.]{Tucsnak2009} to conclude that $A$ is skew-adjoint on $\mathcal{X}$. Let us show that both $(I-A)$ and $(I+A)$ are onto.}

{\textbf{\textbullet $I-A$ is onto:}}

{Let $f^1 \in \mathcal{X}^1$ and $f^2 \in \mathcal{X}^2$, we are seeking for a solution $v = \begin{pmatrix} v^1 \\ v^2 \end{pmatrix} \in \mathcal{X}_1 := \mathcal{X}^1_1 \times \mathcal{X}^2_1 := \ker \gamma^1 \times \ker \gamma^2$ to:
\begin{equation}
\label{eq:systemI-A}
\left\lbrace
\begin{array}{rcl}
v^1 + K v^2 &=& f^1, \\
v^2 - L v^1 &=& f^2.
\end{array}
\right.
\end{equation}
Let us first assume that~\eqref{eq:systemI-A} admits a solution. Then, for all $\varphi^1 \in \mathcal{X}^1_1$, one has:
$$
\begin{array}{rcl}
&& \left( v^1, \varphi^1 \right)_{\mathcal{X}^1} + \left( K v^2, \varphi^1 \right)_{\mathcal{X}^1} = \left( f^1, \varphi^1 \right)_{\mathcal{X}^1}, \\
&\overset{\eqref{eq:formal-adjoint}\text{ with $\varphi^1 \in \mathcal{X}^1_1$ and $v^2 \in \mathcal{X}^2_1$}}{\Longleftrightarrow}& \left( v^1, \varphi^1 \right)_{\mathcal{X}^1} + \left( v^2, L \varphi^1 \right)_{\mathcal{X}^2} = \left( f^1, \varphi^1 \right)_{\mathcal{X}^1}, \\
&\overset{v^2 = L v^1 + f^2}{\Longleftrightarrow}& \left( v^1, \varphi^1 \right)_{\mathcal{X}^1} + \left( L v^1, L \varphi^1 \right)_{\mathcal{X}^2} = \left( f^1, \varphi^1 \right)_{\mathcal{X}^1} - \left( f^2, L \varphi^1 \right)_{\mathcal{X}^2}.
\end{array}
$$
Since $\gamma^1$ is continuous from $\mathcal{Z}^1$ in $\mathcal{U}^1$, its kernel $\mathcal{X}^1_1$ is a closed subspace of $\mathcal{Z}^1$ and inherits of the graph norm of $L$. Hence the above equality reads:
\begin{equation}
\label{eq:RieszI-A}
\left( v^1, \varphi^1 \right)_{\mathcal{Z}^1} = \left( f^1, \varphi^1 \right)_{\mathcal{X}^1} - \left( f^2, L \varphi^1 \right)_{\mathcal{X}^2}, \qquad \forall \varphi^1 \in \mathcal{X}^1_1.
\end{equation}
In summary, \eqref{eq:systemI-A} implies \eqref{eq:RieszI-A}.}

{Applying Riesz representation theorem, there exists a unique $v^1 \in \mathcal{X}^1_1$ satisfying~\eqref{eq:RieszI-A}.}

{Consider the linear continuous extension $\widetilde{K} \in \mathcal{L}(\mathcal{X}^2, (\mathcal{X}^1_0)')$ of $K \in \mathcal{L}(\mathcal{Z}^2, \mathcal{X}^1)$ defined thanks to~\eqref{eq:formal-adjoint} restricted to $\mathcal{X}^1_0 \times \mathcal{Z}^2$:
$$
\left\langle \widetilde{K} \varphi^2, \varphi^1 \right\rangle_{(\mathcal{X}^1_0)',\mathcal{X}^1_0}
:= \left( \varphi^2, L \varphi^1 \right)_{\mathcal{X}^2},
\qquad \forall \varphi^1 \in \mathcal{X}^1_0, \, \varphi^2 \in \mathcal{X}^2,
$$
where the dual $(\mathcal{X}^1_0)'$ is taken with respect to the pivot space $\mathcal{X}^1$.}

{Now, let us denote $v^2 := L v^1 + f^2 \in \mathcal{X}^2$, then:
$$
\begin{array}{rcl}
\left\langle \widetilde{K} v^2, \varphi^1 \right\rangle_{(\mathcal{X}^1_0)',\mathcal{X}^1_0} 
&=& \left( v^2, L \varphi^1 \right)_{\mathcal{X}^2}, \\
&=& \left( L v^1, L \varphi^1 \right)_{\mathcal{X}^2} + \left( f^2, L \varphi^1 \right)_{\mathcal{X}^2}, \\
&\overset{v^1\text{ solution of \eqref{eq:RieszI-A}}}{=}& \left( f^1, \varphi^1 \right)_{\mathcal{X}^1} - \left( v^1, \varphi^1 \right)_{\mathcal{X}^1}, 
\end{array}
$$
By density of $\mathcal{X}^1_0$ in $\mathcal{X}^1$, the right-hand side extends to all $\varphi^1 \in \mathcal{X}^1$, which implies that $\widetilde{K} v^2 = K v^2 \in \mathcal{X}^1$, hence $v^2 \in \mathcal{Z}^2$ and $v^1 + K v^2 = f^1$ in $\mathcal{X}^1$.}

{It remains to verify that indeed $\gamma^2 v^2 = 0$ in $\mathcal{U}^2$. From~\eqref{eq:formal-adjoint}, for all $\varphi^1 \in \mathcal{X}^1_1$, one has:
$$
\begin{array}{rcl}
\left\langle \beta^1 \varphi^1, \gamma^2 v^2 \right\rangle_{(\mathcal{U}^2)',\mathcal{U}^2} 
&=& \left( \varphi^1, K v^2 \right)_{\mathcal{X}^1} - \left( L \varphi^1, v^2 \right)_{\mathcal{X}^2}, \\
&=& \left( \varphi^1, f^1 \right)_{\mathcal{X}^1} - \left( \varphi^1, v^1 \right)_{\mathcal{X}^1} - \left( L \varphi^1, L v^1 \right)_{\mathcal{X}^2} - \left( L \varphi^1, f^2 \right)_{\mathcal{X}^2},\\
&=& \overline{- \left( v^1, \varphi^1 \right)_{\mathcal{Z}^1} + \left( f^1, \varphi^1 \right)_{\mathcal{X}^1} - \left( f^2, L \varphi^1 \right)_{\mathcal{X}^2}},\\
&=& 0,
\end{array}
$$
because $v^1$ is solution of~\eqref{eq:RieszI-A}. Since $\beta^1$ restricted to $\mathcal{X}^1_1$ is assumed to be onto, this shows that $\gamma^2 v^2 = 0$ in $\mathcal{U}^2$, {\it i.e.} that $v^2 \in \mathcal{X}^2_1 := \ker \gamma^2$.}

{Hence for all $f^1 \in \mathcal{X}^1$ and $f^2 \in \mathcal{X}^2$, we found $v^1 \in \mathcal{X}^1_1$ and $v^2 \in \mathcal{X}^2_1$ solution of~\eqref{eq:systemI-A}, showing that $I-A$ is indeed onto.}

{\textbf{\textbullet $I+A$ is onto:} The same proof adapts straightforwardly.}

{We conclude by applying \cite[Proposition~3.7.3.]{Tucsnak2009}.}
\end{proof}

The main result of this section is the following infinite-dimensional analogous of Proposition~\ref{prop:extended-finite-dim}.
\begin{theorem}\label{th:BCS2SD}
With the notations, definitions and assumptions of the beginning of this section, assume furthermore that:
\begin{itemize}
\item[(A1)]
$\gamma^i \in \mathcal{L}(\mathcal{Z}^i,\mathcal{U}^i)$ is onto, $i=1, 2$;
\item[(A2)]
$\mathcal{X}^i_1 := \ker \gamma^i$ is dense in $\mathcal{X}^i$, $i=1, 2$;
{\item[(A3)]
either $\mathcal{X}^1_0 := \ker \gamma^1 \cap \ker \beta^1$ is dense in $\mathcal{X}^1$ and $\beta^1$ restricted to $\mathcal{X}^1_1$ is onto, \\
or $\mathcal{X}^2_0 := \ker \gamma^2 \cap \ker \beta^2$ is dense in $\mathcal{X}^2$ and $\beta^2$ restricted to $\mathcal{X}^2_1$ is onto.}
\end{itemize}

Let us denote $A|_{\mathcal{X}}$ the continuous extension of $A$ to $\mathcal{X}$, with value in $\mathcal{X}_{-1}$, the completion of $\mathcal{X}$ endowed with the norm $\left\| (I-A)^{-1} \cdot \right\|_{\mathcal{X}}$.

Denote furthermore $\mathcal{X}^i_{-1}$ the projection of $\mathcal{X}_{-1}$ on the $i$-th component, for $i=1,2$.
%Let us denote $\mathcal{X}_{-1}^i$ the dual of $\mathcal{X}_1^i := \ker \gamma^i$ with respect to the pivot space $\mathcal{X}^i$, $i=1, 2$.

Then there exist:
\begin{itemize}
\item
a unique control operator $B^1 \in \mathcal{L} (\mathcal{U}^1, \mathcal{X}_{-1}^2)$ associated to\footnote{By ``associated to'', it is meant that $B^1$ is constructed from the operator $\gamma^1$.} $\gamma^1$;
\item
a unique control operator $B^2 \in \mathcal{L} (\mathcal{U}^2, \mathcal{X}_{-1}^1)$ associated to $\gamma^2$;
%\item
%an observation operator $C^1 \in \mathcal{L} (\mathcal{Z}^1, (\mathcal{U}^2)')$;
%\item
%an observation operator $C^2 \in \mathcal{L} (\mathcal{Z}^2, (\mathcal{U}^1)')$;
\end{itemize}
such that the graph of $\mathcal{J} \in \mathcal{L}(\mathcal{E},\mathcal{F})${, $\mathcal{F} := \mathcal{E}'$,} defined by:
\begin{multline}\label{eq:extendedInterconnectionOperator}
\mathcal{J} := \begin{bmatrix}
A|_{\mathcal{X}} & \begin{bmatrix} 0 & B^2 \\ B^1 & 0 \end{bmatrix} \\
- \begin{bmatrix} 0 & {\beta^2} \\ {\beta^1} & 0 \end{bmatrix} & \begin{bmatrix} 0 & 0 \\ 0 & 0 \end{bmatrix} \\
\end{bmatrix}, \\
\mathcal{E} := \left\{ \begin{pmatrix} e^1 \\ e^2 \\ u^1 \\ u^2 \end{pmatrix} \in \mathcal{X}^1 \times \mathcal{X}^2 \times \mathcal{U}^1 \times \mathcal{U}^2 \, \mid \, A|_{\mathcal{X}} \begin{pmatrix} e^1 \\ e^2 \end{pmatrix} + \begin{bmatrix} 0 & B^2 \\ B^1 & 0 \end{bmatrix} \begin{pmatrix} u^1 \\ u^2 \end{pmatrix} \in \mathcal{X}^1\times\mathcal{X}^2 \right\},
\end{multline}
is a Stokes-Dirac structure $\mathcal{D}$ in the bond space $\mathcal{B} := \mathcal{F} \times \mathcal{E}$.%, where $\mathcal{F} := \mathcal{E}'$.
%$$
%\mathcal{E} := \mathcal{X}^1 \times \mathcal{X}^2 \times \mathcal{U}^1 \times \mathcal{U}^2, \qquad
%\mathcal{F} = \mathcal{X}^1 \times \mathcal{X}^2 \times (\mathcal{U}^1)' \times (\mathcal{U}^2)'.
%$$
\iffalse
Furthermore, for all operators $C^1$ and $C^2$ such that $\mathcal{J}$ as above satisfies~\eqref{eq:skew-sym}, the following abstract Green's formula holds:
\begin{equation}\label{eq:PHS-abstract-Green}
\Re{\rm e} \left( L z^1, z^2 \right)_{\mathcal{X}^2}
= \Re{\rm e} \left( z^1, K z^2 \right)_{\mathcal{X}^1}
+ \Re{\rm e} \left\langle \gamma^1 z^1, C^2 z^2 \right\rangle_{\mathcal{U}^1,(\mathcal{U}^1)'}
+ \Re{\rm e} \left\langle \gamma^2 z^2, C^1 z^1 \right\rangle_{\mathcal{U}^2,(\mathcal{U}^2)'}.
\end{equation}
\fi
\end{theorem}

\begin{proof}
The complete proof is postponed to {Appendix}~\ref{app:proof-th-BCS2SD}. It consists of the 3 following steps:
\begin{enumerate}
\item prove that $\left( J, G \right)$ is a boundary control system on $\left(\mathcal{Z}^1\times\mathcal{Z}^2, \mathcal{X}^1\times\mathcal{X}^2, \mathcal{U}^1\times\mathcal{U}^2\right)$;
\item {prove that $\mathcal{J}$ satisfies~\eqref{eq:skew-sym};}%construct an output accordingly to the hypothesis of Theorem~\ref{th:Stokes-Dirac};
\item {prove that the control operator associated to $\left( J, G \right)$ is of the form $\begin{bmatrix} 0 & B^2 \\ B^1 & 0 \end{bmatrix}$, with $B^i$ associated to $\gamma^i$, $i=1,2$}.
\end{enumerate}
\end{proof}

\iffalse
\begin{remark}
In practice, the operators $C^1$ and $C^2$ are not computed using the abstract formulation obtained in the proof of Theorem~\ref{th:BCS2SD}. Indeed, it is proved that they do not consist only of $(B^1)^\star$ and $(B^2)^\star$, {\it i.e.}:
$$
\begin{bmatrix} 0 & C^2 \\ C^1 & 0 \end{bmatrix} \neq \begin{bmatrix} 0 & (B^1)^\star \\ (B^2)^\star & 0 \end{bmatrix}.
$$
The very first problem of the right-hand side being that it is not defined on $\mathcal{Z} = \mathcal{Z}^1\times\mathcal{Z}^2$.

However, integration by parts and comparison with~\eqref{eq:formal-adjoint} allow easy identifications with {the trace operators $\beta^1$ and $\beta^2$}, as will be seen in Section~\ref{sExamples}. %Furthermore, the abstract Green's formula~\eqref{eq:PHS-abstract-Green} shows that the difference between distinct control operators leads to a purely imaginary term.
\end{remark}
\fi

\begin{remark}
It is important to notice that $A|_{\mathcal{X}}$ is the linear continuous extension to $\mathcal{X}$ of $J$ restricted to $\mathcal{X}_1$, {\it i.e.} $A$, which is {\em a priori not} identifiable with $J$ defined on $\mathcal{Z}^1 \times \mathcal{Z}^2$. Indeed, they differ on $(\mathcal{X}^1_1)^{\perp_{\mathcal{Z}^1}} \times (\mathcal{X}^2_1)^{\perp_{\mathcal{Z}^2}}$ since $\mathcal{X}_1^i$ is not dense in $\mathcal{Z}^i$, $i=1, 2$, in general. This difference is exactly the way the control operators are exhibited. See the proof of \cite[Proposition~10.1.2]{Tucsnak2009} for more details. 
\end{remark}

\begin{remark}
In this work, we do not consider non-zero feedthrough operator $D \in \mathcal{L}(\mathcal{U},\mathcal{U}')$ in $\mathcal{J}$. However, as soon as $D$ satisfies $\Re {\rm e} \left\langle D u, u \right\rangle_{\mathcal{U}', \mathcal{U}} = 0$ for all $u \in \mathcal{U}$, the results would follow as well. 
\end{remark}

\subsection{Well-posed linear port-Hamiltonian systems}\label{sec:wplpHs}

This section is devoted to the problem of existence and uniqueness of solution in the particular case where the Hamiltonian is a quadratic form, and without resistive port, {\it i.e.} when $\mathcal{Z} = \mathcal{Z}^1 \times \mathcal{Z}^2$ and $\mathcal{R} = \emptyset$.

This is clearly a restrictive case, however sufficient for lossless pHs of Section~\ref{sExamples}.

\begin{theorem}[Well-posed linear port-Hamiltonian system]\label{th:well-posed}
Let us consider a port-Hamiltonian system as in Proposition~\ref{prop:power-balance-dim-inf}, whose Stokes-Dirac structure is given as in Theorem~\ref{th:BCS2SD}.
Assume furthermore that the Hamiltonian $\mathcal{H}$ is given by a self-adjoint positive-definite operator $Q \in \mathcal{L}(\mathcal{X})$ as $\mathcal{H}(\alpha) = \frac{1}{2} \left( \alpha, Q \alpha \right)_{\mathcal X}$, with $\mathcal{X} = \mathcal{X}^1 \times \mathcal{X}^2$, $\mathcal{R} = \emptyset$, and $\mathcal{U} = \mathcal{U}^1 \times \mathcal{U}^2$.

Then it holds: for all $\alpha_0 \in \mathcal{X}$, and all $u \in H^2_{\rm \ell oc}([0,\infty); \mathcal{U})$ such that $\begin{pmatrix} Q \alpha_0 \\ u(0) \end{pmatrix} \in \mathcal{E}$, there exists a unique trajectory satisfying:
$$
\left( \begin{pmatrix}
\dot{\alpha}(t) \\ -y(t)
\end{pmatrix} , \begin{pmatrix}
Q \alpha(t) \\ u(t)
\end{pmatrix} \right) \in \mathcal{C}([0,\infty); \mathcal{D}),
\quad \text{ with } \quad
\alpha(0) = \alpha_0.
$$
Such a system is said to be a \emph{well-posed linear port-Hamiltonian system}.
\end{theorem}

\begin{proof}
Uniqueness is clear by linearity.

Let us denote:
$$
\mathcal{J} = \begin{bmatrix} A|_{\mathcal{X}} & B \\ -C & 0 \end{bmatrix}.
$$
{
 where $B = \begin{bmatrix}
     0 & B^2 \\ B^1 & 0
 \end{bmatrix}$ and $C = \begin{bmatrix}
     0 & \beta^2 \\ \beta^1 & 0
 \end{bmatrix}$.}
From Theorem~\ref{th:skew-adjoint}, $A$ is skew-adjoint on $\mathcal{X}$. Since $Q$ is bounded, self-adjoint and positive-definite on $\mathcal{X}$, $Q^\frac{1}{2} A Q^\frac{1}{2}$ is also skew-adjoint on $\mathcal{X}$, with domain $Q^{-\frac{1}{2}} \mathcal{X}_1$. Therefore, it is the generator of a strongly continuous group on $\mathcal{X}$~\cite[Theorem~3.8.6.]{Tucsnak2009}.

It is clear that $Q^\frac{1}{2} B \in \mathcal{L}(\mathcal{U},Q^\frac{1}{2} \mathcal{X}_{-1})$.

For all $\begin{pmatrix} Q \alpha_0 \\ u(0) \end{pmatrix} \in \mathcal{E}$, $A Q \alpha_0 + B u(0) \in \mathcal{X}$. Denoting $z_0 := Q^\frac{1}{2} \alpha_0$ and multiplying by $Q^\frac{1}{2}$ gives $Q^\frac{1}{2} A Q^\frac{1}{2} z_0 + Q^\frac{1}{2} B u(0) \in \mathcal{X}$.

Applying \cite[Proposition~4.2.11]{Tucsnak2009}, there exists a unique solution $z$ to:
$$
\dot z(t) = Q^\frac{1}{2} A Q^\frac{1}{2} z(t) + Q^\frac{1}{2} B u(t), \qquad z(0) = z_0 := Q^\frac{1}{2} \alpha_0,
$$
that satisfies $z \in \mathcal{C}^1([0,\infty);\mathcal{X})$.

Defining $\alpha := Q^{-\frac{1}{2}} z$, one has:
$$
\dot \alpha(t) = A Q \alpha(t) + B u(t), \qquad \alpha(0) = \alpha_0,
$$
satisfying $\alpha \in \mathcal{C}^1([0,\infty);\mathcal{X})$.

To conclude, $(I-A) Q \alpha(t) = Q \alpha(t) - \dot \alpha(t) + B u(t)$ for all $t\ge0$ implies that $(I-A) Q \alpha \in \mathcal{C}([0,\infty);\mathcal{X} + B \mathcal{U})$. In other words, $Q \alpha \in \mathcal{C}([0,\infty);\mathcal{Z})$ by Proposition~\ref{prop:BCS}, point~4. In particular the observation operator $C \in \mathcal{L}(\mathcal{Z},\mathcal{U}')$ can be applied to $Q \alpha$ and $y \in \mathcal{C}([0,\infty);\mathcal{U}')$.

All together, and since $H^2_{\rm \ell oc}([0,\infty); \mathcal{U}) \subset \mathcal{C}([0,\infty);\mathcal{U})$, the result follows.
\end{proof}

\begin{remark}
The regularity assumption on $u$ can be relaxed with the less stringent condition $u \in H^1_{\rm \ell oc}([0,\infty); \mathcal{U})$,  provided that $Q^\frac12 B$ is an admissible control operator for the semi-group generated by $Q^\frac{1}{2} A Q^\frac{1}{2}$. See \cite[Chapter~4.]{Tucsnak2009} for more details.
\end{remark}

%%%%%%%%%%%%%%%%%%%%%%%%%%%%%%%%%%%%%
\section{Some useful examples}
\label{sExamples}

This section provides four examples dealing with different differential operators. For the sake of completeness, a case of lossy pHs is included: the electrodynamical problem with Joule's effect, although well-posedness has not been proved for this case.

\subsection{Scalar wave: the $(\div, -\grad)$ case}
\label{s-divgrad}

Let us begin with the classical scalar wave equation, defined on a bounded set $\Omega \subset \mathbb{R}^3$. The governing PDE reads:
$$
\rho \frac{\partial^2 w}{\partial t^2} = \div \left( \bm{T} \grad (w) \right),
$$
where $w$ denotes the deflection from the equilibrium, $\rho$ is the mass density, bounded from above and below, and $\bm{T}$ is Young's modulus, a rank 2 tensor field, symmetric and positive-definite almost everywhere.

Choosing the total mechanical energy, kinetic plus potential, as Hamitlonian, one has to select the \emph{energy variables} to express it. Let us take the linear momentum and the strain:
$$
\alpha^1 := \rho \frac{\partial w}{\partial t}, \qquad \bm{\alpha}^2 := \grad (w).
$$
The Hamiltonian functional is then a quadratic form in these variables:
$$
\mathcal{H} = \frac{1}{2} \int_\Omega \left\{ \frac{1}{\rho} \left( \alpha^1 \right)^2 + (\bm{T} \bm{\alpha}^2) \cdot \bm{\alpha}^2 \right\} {\rm d}\Omega.
$$
The \emph{co-energy variables} are given by the variational derivatives of $\mathcal{H}$ with respect to the energy variables:
$$
e^1 := \frac{\delta \mathcal{H}}{\delta \alpha^1} = \frac{\partial w}{\partial t}, \qquad \bm{e}^2 :=  \frac{\delta \mathcal{H}}{\delta \bm{\alpha}^2} = \bm{T} \grad(w),
$$
that is the velocity and the stress respectively.

Assuming smooth solutions, the power-balance satisfied by the Hamiltonian reads:
$$
\frac{{\rm d} \mathcal{H}}{{\rm d}t} = \left\langle \bm{e}^2 \cdot \bm{n}, e^1 \right\rangle_{H^{-\frac{1}{2}}(\partial\Omega),H^\frac{1}{2}(\partial\Omega)},
$$
giving us informations on boundary controls and obervations that are allowed in the formalism, which must leads to trajectories lying in the Stokes-Dirac structure according to Proposition~\ref{prop:power-balance-dim-inf}.

%As mixed boundary condition implies involved boundary functional spaces for compatibility (if any), l
Let us choose a simple causality and control the velocity at the boundary. The port-Hamiltonian formulation then reads:
$$
\begin{array}{c}
\displaystyle \frac{\partial}{\partial t} \begin{pmatrix} \alpha^1 \\ \bm{\alpha}^2 \end{pmatrix} 
= 
\begin{bmatrix} 0 & \div \\ \grad & \bm{0} \end{bmatrix} 
\begin{pmatrix} e^1 \\ \bm{e}^2 \end{pmatrix}, \\
\displaystyle \begin{pmatrix} u^1 \\ u^2 \end{pmatrix}
=
\begin{bmatrix} \gamma_0 & \bm{0} \\ 0 & \bm{0} \end{bmatrix}
\begin{pmatrix} e^1 \\ \bm{e}^2 \end{pmatrix},
\end{array}
$$
where $\gamma_0$ denotes the Dirichlet trace operator.

In this first example, the spaces and operators defining the operator $J$ are as follows:
$$
\begin{array}{c}
L = \grad, \qquad \mathcal{X}^1 = L^2(\Omega), \qquad \mathcal{Z}^1 = H^1(\Omega), \\
K = -\div, \qquad \mathcal{X}^2 = L^2(\Omega; \mathbb{R}^3), \qquad \mathcal{Z}^2 = H^{\div}(\Omega;\mathbb{R}^3),
\end{array}
$$
where the following Sobolev spaces have been used:
$$
\begin{array}{c}
H^1(\Omega) = \left\{ v \in L^2(\Omega) \mid \grad(v) \in L^2(\Omega; \mathbb{R}^3) \right\}, \\
H^{\div}(\Omega;\mathbb{R}^3) = \left\{ \bm{v} \in L^2(\Omega; \mathbb{R}^3) \mid \div(\bm{v}) \in L^2(\Omega) \right\}.
\end{array}
$$
From de Rham cohomology, it is known~\cite[Chapter~3.]{Monk2003} that the following complex holds:
\begin{equation}\label{eq:deRham-usual}
H^1(\Omega) / \mathbb{R}
\overset{\grad}{\longrightarrow} 
H^{\rm curl}(\Omega;\mathbb{R}^3)
\overset{\rm curl}{\longrightarrow} 
H^{\div}(\Omega;\mathbb{R}^3)
\overset{\div}{\longrightarrow} 
L^2(\Omega).
\end{equation}
Hence, $L$ and $K$ are indeed closed and densely defined as expected.
The spaces and boundary operators defining $G$ are given as follows:
$$
\begin{array}{c}
\gamma^1 = \gamma_0,
\qquad \mathcal{U}^1 = H^\frac{1}{2}(\partial\Omega), \\
\gamma^2 = \bm{0},
\qquad \mathcal{U}^2 = \{ 0 \},
\end{array}
$$
where by definition, $H^\frac{1}{2}(\partial\Omega) \simeq {\rm Ran}\gamma_0$, hence $\gamma^1$ is trivially surjective.

Now, it is a well-known result that $H_0^1(\Omega) := {\rm ker} \gamma_0$ is dense in $L^2(\Omega)$.

Finally, thanks to the (usual) Green's formula, $L$ and $K$ are formal adjoints with respect to $G$, and $C$ is identified as the normal trace operator $\gamma_\perp := \bm{n} \cdot \gamma_0 : H^{\div}(\Omega;\mathbb{R}^3) \rightarrow H^{-\frac{1}{2}}(\partial\Omega) = (\mathcal{U}^1)'$, where $\bm{n}$ is the outward unit normal to the boundary.

Thus, {by} virtue of Theorem~\ref{th:BCS2SD}, {the operators} $L$, $K$ and $G$ generate a Stokes-Dirac structure.

If furthermore we define the multiplicative operator of constitutive relations:
$$
Q := \begin{bmatrix}
\rho^{-1} & \bm{0} \\ 0 & \bm{T}
\end{bmatrix},
$$
then Theorem~\ref{th:well-posed} proves that the scalar wave problem with velocity boundary control is a well-posed linear port-Hamiltonian system, with the normal trace of the stress as {collocated} boundary observation.

\subsection{Three-dimensional elasticity: the $(\Div, -\Grad)$ case}
\label{s-DivGrad}
We consider the linear elastodynamics problem, described by the vector-valued PDE defined on the bounded set $\Omega \subset \bbR^3$:
\begin{equation}
\begin{aligned}
    \rho \diffp[2]{\bm{u}}{t} &= \Div(\bm\Sigma), \\
    \bm\Sigma &= \bm{\mathcal{D}}\bm{\varepsilon},  \\
    \bm{\varepsilon} &= \Grad \bm{u},    \\
\end{aligned} \qquad
\begin{aligned}
\bm{u}: \text{ displacement field},\\
\bm\Sigma \text{ Cauchy stress tensor}, \\
\bm{\varepsilon}: \text{ infinitesimal strain tensor},
\end{aligned}
\end{equation}
where $\rho$ is the mass density and the stiffness tensor $\bm{\mathcal{D}} : \bbR^{3\times 3}_{\text{sym}} \rightarrow \bbR^{3\times 3}_{\text{sym}}$ is a rank 4 tensor that is bounded, symmetric and positive definite almost everywhere. The operator $\Div$ is the columnwise divergence of a tensor field, whereas $\Grad:=\frac{1}{2}(\nabla + \nabla^{\top})$ is the symmetric gradient. This system of equations is formulated as a port-Hamiltonian system by selecting as energy variables the linear momentum and the strain tensor:
\begin{equation*}
    \bm{\alpha}^1 := \rho \diffp{\bm{u}}{t}, \qquad \bm{A}^2 := \bm\varepsilon.
\end{equation*}
The Hamiltonian functional is quadratic in these variables:
\begin{equation}
\mathcal{H} = \energy{\frac{1}{\rho} \norm{\bm{\alpha}^1}^2 + (\bm{\mathcal{D}} \bm{A}^{2}) \cddot  \bm{A}^{2}},
\end{equation}
where $\bm{A} \cddot \bm{B} = \sum_{ij} A_{ij} B_{ij}$ denotes the tensor contraction. The co-energy variables are given by the variational derivative of $\mathcal{H}$ (see~\cite{Brugnoli2019a} for the definition {of the variational derivative} in the tensorial case):
\begin{equation}
\bm{e}^1 := \diffd{\mathcal{H}}{\bm{\alpha}^1} = \diffp{\bm{u}}{t}, \qquad \bm{E}^2 := \diffd{\mathcal{H}}{\bm{A}^{2}} = \bm{\Sigma}.
\end{equation}  
Since the characterization of mixed control spaces for elasticity is involved, we consider for simplicity the case of a uniform boundary control of the normal trace of the Cauchy stress tensor $\bm{\Sigma}$. The port Hamiltonian formulation including the boundary input then reads (cf.  \cite[page 40]{Brugnoli2020PHD}):
\begin{equation}\label{eq:pHsysElas}
\begin{aligned}
\displaystyle
\diffp{}{t}
\begin{pmatrix}
\bm{\alpha}^1 \\
\bm{A}^2
\end{pmatrix} &= \underbrace{
\begin{bmatrix}
\bm{0} & \Div \\
\Grad & \bm{0} \\
\end{bmatrix}}_{J}
\begin{pmatrix}
\bm{e}^1 \\
\bm{E}^2
\end{pmatrix}, \vspace{3pt}\\
\begin{pmatrix}
\bm{u}^1 \\ \bm{u}^2
\end{pmatrix}
 &= \underbrace{
	\begin{bmatrix}
	\bm{0} & \bm{0} \\
	\bm{0} & \bm\gamma_{\perp} \\
	\end{bmatrix}}_{G} \begin{pmatrix}
\bm{e}^1 \\
\bm{E}^2
\end{pmatrix},
\end{aligned}
\end{equation}
where  $\bm\gamma_{\perp}$ denotes the normal trace of a tensor field over the boundary, namely $\bm\gamma_{\perp} :=  \bm{E}^2 = \bm{E}^2 \cdot \bm{n}\vert_{\partial\Omega}$. For this case, the spaces and operators are as follows:
\begin{equation}
\begin{aligned}
    {L} &= \Grad{}, \\
    K &= -\Div{},
\end{aligned} \qquad
\begin{aligned}
    \mathcal{X}^1 &= L^2(\Omega, \bbR^3), \\
    \mathcal{X}^2 &= L^2(\Omega, \bbR^{3\times 3}_{\text{sym}}),
\end{aligned} \qquad
\begin{aligned}
    \mathcal{Z}^1 &= H^{\Grad}(\Omega, \bbR^3), \\
    \mathcal{Z}^2 &= H^{\Div}  (\Omega, \bbR^{3\times 3}_{\text{sym}}),
\end{aligned}
\end{equation}
where the following Sobolev spaces have been introduced:
\begin{equation}
\begin{aligned}
    H^{\Grad}(\Omega, \bbR^3) &= \{\bm{v} \in L^2(\Omega, \bbR^3) \; \mid \; \Grad \bm{v} \in L^2(\Omega, \bbR^{3\times 3}_{\text{sym}})\}, \\
    H^{\Div}(\Omega, \bbR^{3\times 3}_{\text{sym}}) &= \{\bm{V} \in L^2(\Omega,\bbR^{3\times 3}_{\text{sym}}) \; \mid \; \Div \bm{V} \in L^2(\Omega, \bbR^{3})\}.
\end{aligned}
\end{equation}
This operator $-\Div: L^2(\Omega, \bbR^{3\times 3}_{\text{sym}}) \rightarrow L^2(\Omega, \bbR^3)$ is a closed densely defined operator with domain $H^{\Div}  (\Omega, \bbR^{3\times 3}_{\text{sym}})$, while $\Grad: L^2(\Omega, \bbR^3) \rightarrow L^2(\Omega, \bbR^{3\times 3}_{\text{sym}})$ is a closed densely defined operator with domain $H^{\Grad}(\Omega, \bbR^3)$. More precisely, these operators are part of the elasticity complex~\cite{Arnold2021}:

\begin{equation*}
   \mathring{H}^{\Grad}(\Omega, \bbR^3)\xrightarrow{\Grad{}} \mathring{H}^{\Rot\Rot^\top}(\Omega, \bbR^{3\times 3}_{\text{sym}})\xrightarrow{\Rot\Rot^\top}  \mathring{H}^{\Div}(\Omega, \bbR^{3\times 3}_{\text{sym}})\xrightarrow{\Div} L^2(\Omega, \bbR^3), 
\end{equation*}
where the homogeneous boundary conditions, denoted by a ($\circ$) above the {functional} space, are defined within each Sobolev space. The corresponding dual domain complex is given by:
\begin{equation*}
   L^2(\Omega, \bbR^3) \xleftarrow{-\Div{}} {H}^{\Div}(\Omega, \bbR^{3\times 3}_{\text{sym}}) \xleftarrow{\Rot\Rot^\top} {H}^{\Rot\Rot^\top}(\Omega, \bbR^{3\times 3}_{\text{sym}})  \xleftarrow{-\Grad} {H}^{\Grad}(\Omega, \bbR^3).
\end{equation*}
%
%For additional details on this complex, the reader can consult \cite{}.

The boundary input operators spaces for this example are the following:
\begin{equation}
\begin{aligned}
    \gamma^1 &= \bm{0}, \\
    \gamma^2 &= \bm{\gamma}_{\perp}, 
\end{aligned} \qquad
\begin{aligned}
    \mathcal{U}^1 &= \{0\}, \\
    \mathcal{U}^2 &= H^{-1/2}(\partial\Omega, \bbR^3).
\end{aligned}
\end{equation}
The space:
\begin{equation*}
    H^{1/2}(\partial\Omega, \bbR^3) := \ran \bm{\gamma}_0\vert_{H^{\Grad}(\Omega, \bbR^{3})}
\end{equation*}
is defined to be the range of the Dirichlet trace\footnote{The bold notation here distinguishes the vectorial case of Elasticity from the scalar case of the wave equation.} $\bm{\gamma}_0$ on the Sobolev space $H^{\Grad}(\Omega, \bbR^{3})$.
The space: 
\begin{equation*}
    H^{-1/2}(\partial\Omega, \bbR^3) \simeq \ran \bm{\gamma}_{\perp}\vert_{H^{\Div}(\Omega, \bbR^{3\times 3}_{\text{sym}})}
\end{equation*}
is isomorphic to the range of the normal trace operator on the space $H^{\Div}(\Omega, \bbR^{3\times 3}_{\text{sym}})$.
The kernel of the trace operator $\gamma_n$ corresponds to the space:
\begin{equation}
    \ker{\bm\gamma_{\perp}} := \mathring{H}^{\Div}(\Omega, \bbR^{3\times 3}_{\text{sym}}) := \{\bm{V} \in H^{\Div}(\Omega, \bbR^{3\times 3}_{\text{sym}}) \; \mid \;  \bm{V} \cdot \bm{n}|_{\partial\Omega} = 0\},
\end{equation}
which is dense in the space $L^2(\Omega,\bbR^{3\times 3}_{\text{sym}})$, since it contains $\mathcal{C}_0^\infty(\Omega,\bbR^{3\times 3}_{\text{sym}})$ , which is dense in $L^2(\Omega,\bbR^{3\times 3}_{\text{sym}})$.

{It is assumed that for the linear elastodynamics problem with boundary control of the normal trace of the Cauchy stress tensor defines a Stokes-Dirac structure, the following Green formula holds:
\begin{equation}\label{eq:greenElasticity}
\left( \Grad e^1, \bm{E}^2 \right)_{L^2(\Omega, \bbR^{3\times 3}_{\text{sym}})}
+ \left( e^1, \Div \bm{E}^2 \right)_{L^2(\Omega, \bbR^{3})} = \left\langle \bm{\gamma}_0 \bm{e}^1, \bm{\gamma}_{\perp} \bm{E}^2 \right\rangle_{H^{1/2}(\partial\Omega),H^{-1/2}(\partial\Omega)}.
\end{equation}
where the observation operator $C^1 = \bm{\gamma}_0$ corresponds to the (vector) Dirichlet trace of the velocity and $C^2=0$. The duality product corresponds to the duality product between $H^{1/2}(\partial\Omega, \bbR^3)$ and $H^{-1/2}(\partial\Omega, \bbR^3)$. By Theorem \ref{th:BCS2SD} the linear elastodynamic problem defines a Stokes-Dirac structure
} 
The multiplicative constitutive operator:
$$
Q := \begin{bmatrix}
\rho^{-1} & \bm{0} \\ \bm{0} & \bm{\mathcal{D}}
\end{bmatrix},
$$
leads to a well-posed linear port-Hamiltonian system by Theorem~\ref{th:well-posed}.

\subsection{Kirchhoff-Love thin plates: the $(\div \Div, \Grad \grad)$ case}
\label{s-Hess}
In this example, we consider the mechanical vibrations of thin plate using the Kirchhoff-Love model, expressed by the following PDE defined on the bounded set $\Omega \subset \bbR^2$:
\begin{equation}
\begin{aligned}
    \mu \diffp[2]{w}{t} &= -\div\Div(\bm{M}), \\
    \bm{M} &= \bm{\mathcal{D}}_b \bm{\kappa},  \\
    \bm{\kappa} &= \Grad\grad{w},    \\
\end{aligned} \qquad
\begin{aligned}
w: \text{ vertical displacement},\\
\bm{M}: \text{ bending momenta tensor}, \\
\bm{\kappa}: \text{ infinitesimal curvature tensor},
\end{aligned}
\end{equation}
where $\mu$ is the mass density per unit area and the bending stiffness tensor $\bm{\mathcal{D}}_b : \bbR^{2\times 2}_{\text{sym}} \rightarrow \bbR^{2\times 2}_{\text{sym}}$ is a rank 4 tensor that is bounded, symmetric and positive definite almost everywhere. {The port-Hamiltonian structure of this PDE can be exposed} if the linear momentum and the strain tensor are selected as energy variables:
\begin{equation*}
    \alpha^1 := \mu \diffp{w}{t}, \qquad \bm{A}^2 := \bm\kappa.
\end{equation*}
The Hamiltonian functional is quadratic in this variables:
\begin{equation}
\mathcal{H} = \energy{\frac{1}{\mu} (\alpha^1)^2 + (\bm{\mathcal{D}}_b \bm{A}^{2}) \cddot  \bm{A}^{2}}.
\end{equation}
The co-energy variables are given by:
\begin{equation}
{e}^1 := \diffd{\mathcal{H}}{\alpha^1} = \diffp{w}{t}, \qquad \bm{\mathcal{E}}^2 := \diffd{\mathcal{H}}{\bm{A}^{2}} = \bm{M}.
\end{equation}  
In this case we consider uniform boundary conditions, given by the linear and angular velocities. The port Hamiltonian formulation reads (cf.  \cite[page 57]{Brugnoli2020PHD}):
\begin{equation}\label{eq:pHsysKir}
\begin{aligned}
\displaystyle
\diffp{}{t}
\begin{pmatrix}
{\alpha}^1 \\
\bm{A}^2
\end{pmatrix} &= \underbrace{
\begin{bmatrix}
0 & -\div\Div \\
\Grad\grad & \bm{0} \\
\end{bmatrix}}_{J}
\begin{pmatrix}
{e}^1 \\
\bm{E}^2
\end{pmatrix}, \vspace{3pt}\\
\begin{pmatrix}
	\bm{u}^1 \\
	\bm{u}^2
	\end{pmatrix}
 &= \underbrace{
	\begin{bmatrix}
	\begin{bmatrix}
	\gamma_{0} \\
	\gamma_{1}
	\end{bmatrix} & \bm{0} \\
	0 & \bm{0}
	\end{bmatrix}}_{G} \begin{pmatrix}
{e}^1 \\
\bm{E}^2
\end{pmatrix},
\end{aligned}
\end{equation}
where $\gamma_{0}$ denotes the trace operator over the boundary and $\gamma_1$ denotes the normal derivative trace, i.e. $\gamma_1 e^1 = \partial_{\bm{n}} e^1\vert_{\partial\Omega}$.  For this case the spaces and operators are as follows:
\begin{equation}
\begin{aligned}
    {L} &= \Grad\grad, \\
    K &= \div\Div,
\end{aligned} \qquad
\begin{aligned}
    \mathcal{X}^1 &= L^2(\Omega), \\
    \mathcal{X}^2 &= L^2(\Omega, \bbR^{2\times 2}_{\text{sym}}),
\end{aligned} \qquad
\begin{aligned}
    \mathcal{Z}^1 &= H^{2}(\Omega), \\
    \mathcal{Z}^2 &= H^{\div\Div}  (\Omega, \bbR^{2\times 2}_{\text{sym}}).
\end{aligned}
\end{equation}
where $\Grad\grad$ corresponds to the Hessian and the following Sobolev spaces have been introduced:
\begin{equation}
\begin{aligned}
    H^{2}(\Omega) &= \{v \in L^2(\Omega) \vert \; \Grad\grad \bm{v} \in L^2(\Omega, \bbR^{2\times 2}_{\text{sym}})\}, \\
    H^{\div\Div}(\Omega, \bbR^{2\times 2}_{\text{sym}}) &= \{\bm{V} \in L^2(\Omega,\bbR^{2\times 2}_{\text{sym}}) \vert \; \div\Div \bm{V} \in L^2(\Omega)\}.
\end{aligned}
\end{equation}
This operator $\div\Div: L^2(\Omega, \bbR^{2\times 2}_{\text{sym}}) \rightarrow L^2(\Omega, \bbR^3)$ is a closed densely defined operator with domain $H^{\div\Div}  (\Omega, \bbR^{2\times 2}_{\text{sym}})$, while $\Grad\grad: L^2(\Omega) \rightarrow L^2(\Omega, \bbR^{2\times 2}_{\text{sym}})$ is a closed densely defined operator with domain $H^{2}(\Omega)$. This is known from the more general fact the Hessian (i.e. the Gradgrad operator) and divDiv operators are part of the Hessian Hilbert complexes and its corresponding adjoint complex, the divDiv complex, respectively \cite{Arnold2021,Pauly2020,Pauly2022}. The Hessian complex in 2 dimension reads:
\begin{equation*}
   \mathring{H}^{2}(\Omega)\xrightarrow{\Grad\grad{}} \mathring{H}^{\Curl}(\Omega, \bbR^{2\times 2}_{\text{sym}})\xrightarrow{\Curl} L^2(\Omega, \bbR^2), \end{equation*}
where the homogeneous boundary conditions are defined within each Sobolev space. The corresponding dual domain complex (the divDiv complex) is given by:
\begin{equation*}
   L^2(\Omega) \xleftarrow{\div\Div{}} {H}^{\div\Div}(\Omega, \bbR^{2\times 2}_{\text{sym}}) \xleftarrow{\text{sym}\Curl} {H}^{1}(\Omega, \bbR^{2}).
\end{equation*}

Since the problem is of second differential order, the input boundary space consists of a cartesian product:
\begin{equation}
\begin{aligned}
    \gamma^1 &= \begin{bmatrix}
    \gamma_0 \\ \gamma_1
    \end{bmatrix}, \qquad 
      \mathcal{U}^1 = H^{3/2}(\partial\Omega) \times H^{1/2}(\partial\Omega), \\
    \gamma^2 &= \bm{0}, \qquad \qquad \mathcal{U}^2 = \{0\}.
     \end{aligned}
\end{equation}
The space $H^{3/2}(\partial\Omega)$ is taken to be the space of traces of functions belonging to $H^2(\Omega)$. The normal derivative trace can be extended as a linear continuous surjective mapping \cite[Th. 3.6.6]{Tucsnak2009}:
\begin{equation}
  \partial_{\bm{n}} : H^2(\Omega) \rightarrow H^{1/2}(\partial\Omega).
\end{equation}
As a consequence, the $G$ operator is surjective. Furthermore the kernel of $\gamma$ corresponds to the space:
\begin{equation}
    \ker\begin{bmatrix}
    \gamma_0 \\ \gamma_1
    \end{bmatrix} = H^2_0(\Omega) := \{v \in H^2 \; \mid \; v\vert_{\partial\Omega}= \partial_{\bm{n}} v\vert_{\partial\Omega} =0 \}
\end{equation}
which is dense in $L^2(\Omega)$, see \cite[Def. 13.4.6 and Prop. 3.6.7.]{Tucsnak2009}. {By assumption, the following Green formula holds:
\begin{equation}\label{eq:greenKirchhoff}
\begin{aligned}
\left( \Grad\grad e^1, \bm{E}^2 \right)_{L^2(\Omega, \bbR^{2\times 2}_{\text{sym}})}
&= \left( e^1, \div\Div \bm{E}^2 \right)_{L^2(\Omega)} \\
&+ \left\langle \gamma_0 \bm{e}^1, \gamma_{1, nn} \bm{E}^2 \right\rangle_{H^{3/2}(\partial\Omega),H^{-3/2}(\partial\Omega)} + \left\langle \gamma_1 {e}^1, {\gamma}_{0, nn} \bm{E}^2 \right\rangle_{H^{1/2}(\partial\Omega),H^{-1/2}(\partial\Omega)}.
\end{aligned}
\end{equation}
This Green formula is also reported in \cite[Th. 2.2.]{Amara2002} for $C^\infty(\Omega, \bbR^{2\times 2}_{\text{sym}})$ tensor fields and $H^2(\Omega)$ vector fields. The observation operator:
\begin{equation*}
    C^2\bm{E}^2 =\begin{bmatrix}
    \gamma_{1, nn} \\
    \gamma_{0, nn}
    \end{bmatrix}\bm{E}^2 =
    \begin{bmatrix}
    -\bm{n} \cdot \Div \bm{E}^2 - \partial_{\bm{t}} (\bm{n}^\top \bm{E}^2 \bm{t})\\
    \bm{n}^\top \bm{E^2}\bm{n}
    \end{bmatrix}
\end{equation*}
correspond to \textit{effective shear force} and \textit{bending momentum} definition, with $\bm{t}$ the unit tangent vector to the boundary. The duality boundary product is here given both by the duality product between $H^{1/2}(\partial\Omega)$ and $H^{-1/2}(\partial\Omega)$, and between $H^{3/2}(\partial\Omega)$ and $H^{-3/2}(\partial\Omega)$.}

So, by virtue of Theorem \ref{th:BCS2SD}, the Kirchhoff-Love model for this plate defines a Stokes-Dirac structure. 
Finally, defining the multiplicative constitutive operator:
$$
Q := \begin{bmatrix}
\mu^{-1} & \bm{0} \\ 0 & \bm{\mathcal{D}}_b
\end{bmatrix},
$$
{by Theorem \ref{th:well-posed}} the Kirchhoff-Love model for thin plate with Dirichlet and Neumann boundary controls of the vertical displacement is a well-posed linear port-Hamiltonian system, with effective shear force and bending momentum as {collocated} observations at the boundary, by Theorem~\ref{th:well-posed}.

\subsection{Maxwell equations: the $(\curl, \curl)$ case}
\label{s-curl3D}

As last example, we propose the Maxwell equations, already treated in~\cite{VanDerSchaft2014,Weiss2013}.

Let us denote $\bm{E}$ and $\bm{B}$ the eletric and magnetic fields respectively of a domain $\Omega\subset\mathbb{R}^3$, and $\bm{D}$ and $\bm{H}$ the respective auxiliary fields \cite{Monk2003}. The governing system is composed of the Maxwell-Ampère and Maxwell-Faraday {dynamical} equations:
$$
\begin{array}{c}
\displaystyle \frac{\partial \bm{D}}{\partial t} - \curl(\bm{H}) = \bm{J},\\
\displaystyle \frac{\partial \bm{B}}{\partial t} + \curl(\bm{E}) = 0,
\end{array}
$$
where $\bm{J}$ is the free current density. Following \cite{VanDerSchaft2002}, we do not consider the two static equations explicitely, namely Maxwell-Gau\ss{} $\div(\bm{D}) = \rho$ in presence of a charge density, or Maxwell-flux $\div(\bm{B}) = 0$.

The total electromagnetic energy is given in terms of the energy variables $\bm{D}$ and $\bm{B}$:
$$
\mathcal{H}(\bm{D}, \bm{B}) = \frac{1}{2} \int_\Omega \left( \frac{\left\|\bm{D}\right\|^2}{\epsilon} + \frac{\left\|\bm{B}\right\|^2}{\mu} \right) {\rm d}\Omega,
$$
where $\epsilon$ is the electric permittivity and $\mu$ the magnetic permeability.

The co-energy variables are then:
$$
\bm{E} := \frac{\delta \mathcal{H}}{\delta \bm{D}} = \frac{1}{\epsilon} \bm{D}, \qquad \bm{H} := \frac{\delta \mathcal{H}}{\delta \bm{B}} = \frac{1}{\mu} \bm{B}.
$$

Thanks to the {following} Green's formula, see e.g. \cite[Theorem~3.31]{Monk2003}:
\begin{equation}\label{eq:Green-curl}
\int_\Omega \bm{U} \cdot \curl(\bm{V}) {\rm d}\Omega = \int_\Omega \bm{V} \cdot \curl(\bm{U}) {\rm d}\Omega - \int_{\partial\Omega} \gamma_0(\bm{U} \wedge \bm{V}) \cdot \bm{n},
\end{equation}
where $\wedge$ denotes the vector product in $\bbR^3$, the electro-magnetic power is computed as:
$$
\frac{{\rm d} \mathcal{H}}{{\rm d} t} = - \int_{\partial\Omega} \bm{\Pi} \cdot \bm{n} - \int_\Omega \bm{E} \cdot \bm{J},
$$
where $\bm{\Pi} := \gamma_0\left( \bm{E} \wedge \bm{H} \right)$ is known as the Poynting vector. Using Ohm's law $\bm{J} = \eta^{-1} \bm{E}$, $\eta$ being the resistivity, the second term of the power balance is negative: $- \int_\Omega \bm{E} \cdot \bm{J} = - \int_\Omega \eta^{-1} \|\bm{E}\|^2 \le 0$. This is actually Joule's effect, {which corresponds to} loss of energy in the thermal domain \cite{Vu2016}.

Regarding the boundary control, we again avoid mixed boundary condition for the sake of simplicity, because of the difficulty lying in the determination of the boundary functional spaces. Let us choose to control the twisted tangential trace of the magnetic field $\bm{u} = \gamma_t(\bm{H}) := \bm{n} \wedge \gamma_0 (\bm{H})$ and to observe the tangential trace of the electric field $\bm{y} = \gamma_T (\bm{E}) := \left( \bm{n} \wedge \gamma_0 (\bm{E}) \right) \wedge \bm{n}$. One may indeed verify that $\bm{u} \cdot \bm{y} = \bm{\Pi} \cdot \bm{n}$.

To summarize, the port-Hamiltonian formulation reads:
$$
\begin{pmatrix}
\partial_t \bm{D} \\ \partial_t \bm{B} \\ \bm{f}_{\bm{J}}
\end{pmatrix}
=
\begin{bmatrix}
\bm{0} & \curl & - I \\ - \curl & \bm{0} & \bm{0} \\ I & \bm{0} & \bm{0}
\end{bmatrix}
\begin{pmatrix}
\bm{E} \\ \bm{H} \\ \bm{e}_{\bm{J}}
\end{pmatrix},
\qquad
\left\{\begin{array}{rcl}
\bm{u} &=& \gamma_t(\bm{H}),\\
\bm{y} &=& \gamma_T(\bm{E}), 
\end{array}\right.
$$
together with the constitutive relations:
$$
\left\{
\begin{array}{rcl}
\bm{E} &=& \epsilon^{-1} \bm{D}, \\
\bm{H} &=& \mu^{-1} \bm{B}, \\
\bm{e}_{\bm{J}} &=& \eta^{-1} \bm{f}_{\bm J}.
\end{array}\right.
$$

Using again the de Rham complex~\eqref{eq:deRham-usual}, one gets closed and densely-defined $L$ and $K$ operators by setting:
$$
\begin{array}{c}
L = \begin{bmatrix} -\curl \\ I \end{bmatrix}, \qquad \mathcal{X}^1 = L^2(\Omega; \mathbb{R}^3), \qquad \mathcal{Z}^1 = H^{\curl}(\Omega; \mathbb{R}^3), \\
K = \begin{bmatrix} \curl & -I \end{bmatrix}, \qquad \mathcal{X}^2 = L^2(\Omega; \mathbb{R}^3) \times L^2(\Omega; \mathbb{R}^3), \qquad \mathcal{Z}^2 = H^{\curl}(\Omega; \mathbb{R}^3) \times L^2(\Omega; \mathbb{R}^3)
.
\end{array}
$$
and
$$
\gamma^1 = \bm{0},
\qquad
\gamma^2 = \begin{bmatrix}
 \gamma_t & \bm{0} \\ \bm{0} & \bm{0}
\end{bmatrix},
\qquad
\mathcal{U}^1 = \{ 0 \},
\qquad
\mathcal{U}^2 = Y(\partial\Omega) \times \{ 0 \},
$$
where:
$$
Y(\partial\Omega) := \left\{ \bm{v} \in H^{-\frac{1}{2}}(\partial\Omega;\bbR^3) \, \mid \, \exists \bm{u} \in H^{\curl}(\Omega;\mathbb{R}^3), \gamma_t(\bm{u}) = \bm{v} \right\}.
$$
We refer to \cite[Chapter~3.]{Monk2003} for more details on $Y(\partial\Omega)$. For our purpose, we only use \cite[Theorem~3.31]{Monk2003}, stating that $G$ is surjective. Furthermore, $\ker\gamma_t = \overline{\mathcal{C}^\infty_0(\Omega;\mathbb{R}^3)}$, the closure being taken in $H^{\curl}(\Omega; \mathbb{R}^3)$, from \cite[Theorem~3.33]{Monk2003}. In particular, the kernel contains $\mathcal{C}^\infty_0(\Omega;\mathbb{R}^3)$, which is dense in $L^2(\Omega;\mathbb{R}^3)$. Hence the kernels of $\gamma^i$, $i=1,2$, are dense in $\mathcal{X}^i$, $i=1,2$, respectively.

{From~\eqref{eq:Green-curl}, Theorem~\ref{th:BCS2SD} applies:} the electro-magnetic problem with {twisted} tangential control of the magnetic field generates a Stokes-Dirac structure on the bond space $\mathcal{B} = \mathcal{E} \times \mathcal{F}$, with:
$$
{\mathcal{E} = \begin{bmatrix} I_{\mathcal{V}} \\ G \end{bmatrix} \mathcal{V},
\qquad
\mathcal{V} := H^{\curl}(\Omega;\mathbb{R}^3) \times H^{\curl}(\Omega;\mathbb{R}^3) \times L^2(\Omega;\mathbb{R}^3),
\qquad
\mathcal{F} = \mathcal{E}'.}% \simeq \mathcal{E},
$$

\begin{remark}
Note that $\mathcal{Z}^i$, $i=1,2$ are not identifiable with the $\mathcal{Z}$ and $\mathcal{R}$ spaces that define the ports formulation used {\it e.g.} in Proposition~\ref{prop:power-balance-dim-inf}. Indeed, the former spaces have been used to prove that $\mathcal{J}$ generates a Stokes-Dirac structure on $\mathcal{B}$, while the latter consider the dynamics of the energy variables $\bm{D}$ and $\bm{B}$.

The splitting of {$\mathcal{V}$} is then chosen accordingly to either the algebraic point of view (considering flows and efforts), or {the dynamical systems} point of view (considering the energy and co-energy variables together with the resistive port). In this example, we either consider:
$$
{\mathcal{V} = \underbrace{H^{\curl}(\Omega;\mathbb{R}^3)}_{\mathcal{Z}^1} \times \underbrace{H^{\curl}(\Omega;\mathbb{R}^3) \times L^2(\Omega;\mathbb{R}^3)}_{\mathcal{Z}^2},}
$$
for the algebraic point of view, or:
$$
{\mathcal{V} = \underbrace{H^{\curl}(\Omega;\mathbb{R}^3) \times H^{\curl}(\Omega;\mathbb{R}^3)}_{\mathcal{Z}} \times \underbrace{L^2(\Omega;\mathbb{R}^3)}_{\mathcal{R}},}
$$
for the dynamical {systems} point of view.
\end{remark}

Note that since $\mathcal{R} \neq \emptyset$, Theorem~\ref{th:well-posed} can not apply {directly} to this example, further work is needed to conclude to the well-posedness of this \emph{constrained} system, as soon as $\eta^{-1} \not\equiv 0$ or $\not\equiv \infty$ (at the physical level,  $\eta^{-1} =0$ models a perfect insulator, whereas  $\eta^{-1} =\infty$ models a perfect conductor). However, this system has already been proved to be well-posed for many kinds of boundary controls. See \textit{e.g.}~\cite{Weiss2013} and references therein.

It is furthermore possible to define the \emph{linear} multiplicative operator $Q := \begin{bmatrix} \epsilon^{-1} & \bm{0} \\ \bm{0} & \mu^{-1} \end{bmatrix}$ relating energy and co-energy variables, as well as the \emph{linear} multiplicative operator $S := \eta^{-1}$ accounting for the resistive constitutive law $\mathcal{S} \subset \mathcal{R} \times \mathcal{R}$, defined by Ohm's law: $\bm{e}_{\bm{J}} = S \bm{f}_{\bm{J}}$.

\subsection{Summary of the examples}

Let us summarize the above examples in Table~\ref{tabrecap}.
\begin{table}[htb]
%\label{tabrecap}
\centering
\begin{tabular}{|c||c|c|c|c|}
  \hline
  \S. & Wave & Elasticity & Kirchoff-Love & Maxwell \\ \hline
  $L$ & $-\grad$ & $-\Grad$ & $\Grad\grad$ & $\begin{bmatrix} -\curl \\ I \end{bmatrix}$ \\
  $K$ & $\div$ & $\Div$ & $\div\Div$ & $\begin{bmatrix} \curl & -I \end{bmatrix}$ \\
  $\mathcal{Z}^1$ & $H^1(\Omega;\mathbb{R})$ & $H^1(\Omega;\mathbb{R}^3)$ & $H^2(\Omega)$ & $H^{\curl}(\Omega; \mathbb{R}^3)$ \\
  $\mathcal{X}^1$ & $L^2(\Omega;\mathbb{R})$ & $L^2(\Omega;\mathbb{R}^3)$ & $L^2(\Omega)$ & $L^2(\Omega;\mathbb{R}^3)$ \\
  $\mathcal{Z}^2$ & $H^{\div}(\Omega;\mathbb{R}^3)$ & $H^{\Div}(\Omega;\mathbb{R}^3)$ & $H^{\div\Div}(\Omega;\mathbb{R}^{2\times 2}_{\text{sym}})$ & $H^{\curl}(\Omega; \mathbb{R}^3) \times L^2(\Omega; \mathbb{R}^3)$ \\
  $\mathcal{X}^2$ & $L^2(\Omega;\mathbb{R}^3)$ & $L^2(\Omega;\mathbb{R}^{3\times 3}_{\text{sym}})$ & $L^2(\Omega;\mathbb{R}^{2\times 2}_{\text{sym}})$ & $L^2(\Omega;\mathbb{R}^3)\times L^2(\Omega;\mathbb{R}^3)$ \\
  $\gamma^1$ & $\gamma_0$ & $\bm{0}$ & $\begin{bmatrix} \gamma_0 \\ \gamma_1  \end{bmatrix} $ & $\bm{0}$ \\
  $\mathcal{U}^1$ & $H^\frac{1}{2}(\partial\Omega;\mathbb{R})$ & $\{0\}$ & $H^{3/2}(\partial\Omega) \times H^{1/2}(\partial\Omega)$ & $\{0\}$ \\
  $C^1$ & $\{0\}$ & $\bm{\gamma}_0$ & $\begin{bmatrix} 0 \\ 0  \end{bmatrix}$& $\bm{\gamma}_T$ \\
  $(\mathcal{U}^2)'$ & $\{ 0 \}$ & $H^{\frac{1}{2}}(\partial\Omega;\mathbb{R}^3)$ & $\{0\}$ & $Y'(\partial\Omega) \times \{0\}$ \\
  $\gamma^2$ & $\bm{0}$ & $\bm\gamma_{\perp}$ & $\bm{0}$ & $\begin{bmatrix} \gamma_t & \bm{0} \\ \bm{0} & \bm{0} \end{bmatrix}$ \\
  $\mathcal{U}^2$ & $\{0\}$ & $H^{-\frac{1}{2}}(\partial\Omega;\mathbb{R}^3)$ & $\{0\}$ & $Y(\partial\Omega) \times \{0\}$\\
  $C^2$ & $\gamma_\perp$ & ${0}$& $\begin{bmatrix} \gamma_{1, nn} \\ \gamma_{0, nn}  \end{bmatrix}$ & $\begin{bmatrix} \bm{0} & \bm{0} \\ \bm{0} & \bm{0} \end{bmatrix}$\\
  $(\mathcal{U}^1)'$ & $H^{-\frac{1}{2}}(\partial\Omega;\mathbb{R})$ & $\{0\}$ & $H^{-3/2}(\partial\Omega) \times H^{-1/2}(\partial\Omega)$ & $\{0\}$ \\
  $Q$ & $\begin{bmatrix} \rho^{-1} & \bm{0} \\ 0 & \bm{T} \end{bmatrix}$ & $\begin{bmatrix} \rho^{-1} & \bm{0} \\ 0 & \bm{\mathcal{D}} \end{bmatrix}$ & $\begin{bmatrix} \rho^{-1} & \bm{0} \\ 0 & \bm{\mathcal{D}}_b \end{bmatrix}$ & $\begin{bmatrix} \epsilon^{-1} & \bm{0} \\ 0 & \mu^{-1} \end{bmatrix}$ \\
  \hline
\end{tabular}
\caption{Examples of linear port-Hamiltonian system.}
\label{tabrecap}
\end{table}

%%%%%%%%%%%%%%%%%%%%%%%%%%%%%%%%%%%%%
\section{Perspectives and Conclusion}\label{sec:conclusion}

The proposed construction shows how linear wave-like systems can be associated to the Stokes-Dirac geometric structure via the theory of boundary control systems. In particular, linear elastodynamic problems fit into this framework. This extends the canonical Stokes-Dirac structure defined in \cite{VanDerSchaft2002}, that includes the case of scalar and electromagnetic waves only. This allows to deduce well-posedness of the considered class of port-Hamiltonian systems. \\

The correct specification of the functional analytic framework is crucial for discretization purposes. For instance, in a finite element context discrete spaces for the variables are chosen in suitable subspaces of the infinite-dimensional functional spaces. The assumed structure of the systems under consideration can be readily discretized using mixed finite element strategies.  In particular, the employment of mixed finite elements for port-Hamiltonian systems has been explored in \cite{cardoso2021pfem}, where it is shown how several systems, linear and non-linear, can be structurally discretized via finite elements. A complete proof of convergence of mixed finite elements with boundary control is presented in \cite{haine2022conv}, where it is proven that several families of finite elements, beyond those satisfying a de Rham subcomplex property, lead to convergence. A geometric viewpoint is instead presented in \cite{brugnoli2022df}, where the discretization mimetically represents the continuous weak formulation via finite element differential forms that constitute a de Rham subcomplex. This approach leads to a primal-dual formulation, capable of retaining the power balance at the discrete level. Finite element exterior calculus provides a framework for the discretization of partial differential equations, unifying concepts from topology, geometry and algebra. For this reason a natural extension of this work would consist in combining the functional analytic setting with the geometric one. \\

The inclusion of unbounded dissipation operators, that occur in e.g. Rayleigh damping, represents another important development of the present work. Furthermore, extending the presented analysis to non linear constitutive relations is also of fundamental for applications. In this case, convexity of the Hamiltonian constitutes a fundamental hypothesis to establish well-posedness.

%\cite{egger2021asymptotic}
\nocite{*}

\bibliographystyle{plain}
\bibliography{mybiblio}

\appendix

\section{Backgrounds on boundary control systems}

Let us start by the definition of a boundary control system, as given in~\cite[Chapter~10]{Tucsnak2009}.

\begin{definition}[Boundary Control Systems]\label{def:BCS}
Let $\mathcal{Z}, \mathcal{X}, \mathcal{U}$ be three complex Hilbert spaces, such that $\mathcal{Z} \subset \mathcal{X}$ with continuous embedding.

Let $J \in \mathcal{L} (\mathcal{Z}, \mathcal{X})$ and $G \in \mathcal{L} (\mathcal{Z}, \mathcal{U})$ be two linear operators.

The couple $(J, G)$ is a boundary control system on $(\mathcal{Z}, \mathcal{X}, \mathcal{U})$ if:
\begin{itemize}
\item[(i)]
$G$ is onto;
\item[(ii)]
$\ker G$ is dense in $\mathcal{X}$;
\end{itemize}
and there exists $\beta \in \bbC$ such that:
\begin{itemize}
\item[(iii)]
$\beta I - J$ restricted to $\ker G$ is onto;
\item[(iv)]
$\ker (\beta I - J) \cap \ker G = \{ 0 \}$.
\end{itemize}
$\mathcal{Z}$ is called the \emph{solution space}, $\mathcal{X}$ the \emph{state space} and $\mathcal{U}$ the \emph{input space}.
\end{definition}

The following Proposition~\ref{prop:BCS} gathers well-known results. Proofs can be found, {\it e.g.}, in~\cite[Chapter~10]{Tucsnak2009} and references therein.

\begin{proposition}\label{prop:BCS}
Let $(J,G)$ be a boundary control system on $(\mathcal{Z}, \mathcal{X}, \mathcal{U})$.

Denote $\mathcal{X}_1 := \ker G$, $A := J|_{\mathcal{X}_1}$, and $\mathcal{X}_{-1}$ the completion of $\mathcal{X}$ endowed with the norm $\left\| (\beta I - A)^{-1} \cdot \right\|_\mathcal{X}$ for some fixed $\beta \in \rho(A)$. Then:
\begin{enumerate}
\item
$\mathcal{X}_1$ is a Hilbert space endowed with the graph norm of $A$, and a continuously embedded closed subspace of $\mathcal{Z}$ (generally not densely embedded);
\item
$A \in \mathcal{L}(\mathcal{X}_1, \mathcal{X})$ and can be continuously extended to an operator $A|_{\mathcal{X}}$ in $\mathcal{L}(\mathcal{X}, \mathcal{X}_{-1})$. Furthermore, if $A$ is skew-adjoint on $\mathcal{X}$, then $A|_{\mathcal{X}}$ is skew-adjoint on $\mathcal{X}_{-1}$;
\item
for $\beta \in \bbC$ as in Definition~\ref{def:BCS}, $\beta \in \rho(A)$, the resolvent set of $A$, and $(\beta I - A)^{-1} \in \mathcal{L}(\mathcal{X},\mathcal{X}_1)$, $(\beta I - A|_{\mathcal{X}})^{-1} \in \mathcal{L}(\mathcal{X}_{-1},\mathcal{X})$.

Furthermore, the graph norm of $A$ on $\mathcal{X}_1$ is equivalent to the norm $\| (\beta I-A) \cdot \|_{\mathcal{X}}$;
\item
there exists a unique control operator $B \in \mathcal{L}(\mathcal{U}, \mathcal{X}_{-1})$ such that:
$$
J = A|_{\mathcal{X}} + BG, \qquad G (\beta I - A|_{\mathcal{X}})^{-1} B = I_\mathcal{U},
$$
furthermore, the operator $\begin{bmatrix} I_{\mathcal{Z}} \\ G \end{bmatrix}$ is a bounded bijection between $\mathcal{Z}$ and $\left\{ \begin{pmatrix} z \\ u \end{pmatrix} \in \mathcal{X} \times \mathcal{U} \; \mid \; A|_{\mathcal{X}} z + B u \in \mathcal{X} \right\}$;
\item
$\mathcal{Z} = (\beta I - A|_{\mathcal{X}})^{-1} \left( \mathcal{X} + B \mathcal{U} \right) = \mathcal{X}_1 + (\beta I - A|_{\mathcal{X}})^{-1} B \mathcal{U}$, and $B$ is strictly unbounded, meaning that $\mathcal{X} \cap B \mathcal{U} = \left\{ 0 \right\}$, and bounded from below. In particular, for all $z\in \mathcal{Z}$, there exists a unique $z_0 \in \mathcal{X}_1$ and a unique $u \in \mathcal{U}$ such that $z = z_0 + ( \beta I - A|_{\mathcal{X}})^{-1} B u$;%. Furthermore,} the sum is direct in $\mathcal{Z}$;% (meaning that the orthogonal $\mathcal{X}_1^{\perp_{\mathcal{Z}}}$ of $\mathcal{X}_1$ in the $\mathcal{Z}$-norm is $(\beta I - \change{A|_{\mathcal{X}}})^{-1} B \mathcal{U}$);
\end{enumerate}
\end{proposition}
%{ON A ENLEVé les items 6, 7, 8 : vérifier qu'on N'en a PAS besoin dans les preuves !}

\section{Proof of Theorem~\ref{th:BCS2SD}}\label{app:proof-th-BCS2SD}

Let us start by showing that $\left( J, G \right)$ is a boundary control system on $\left( \mathcal{Z}^1\times\mathcal{Z}^2, \mathcal{X}^1\times\mathcal{X}^2, \mathcal{U}^1\times\mathcal{U}^2 \right)$.

The four points of Definition~\ref{def:BCS} have to be checked.

\noindent\textbf{Point~$(i)$:} Since $\gamma^i \mathcal{Z}^i = \mathcal{U}^i$, $i=1, 2$, by assumption $(A1)$, $G (\mathcal{Z}^1\times\mathcal{Z}^2) = \mathcal{U}^1 \times \mathcal{U}^2 =: \mathcal{U}$, {\it i.e.} point~$(i)$ of Definition~\ref{def:BCS} holds.

\noindent\textbf{Point~$(ii)$:} Since $\mathcal{X}_1 := \ker G = \left\lbrace \begin{pmatrix} e^1 \\ e^2 \end{pmatrix} \in \mathcal{Z} \; \mid \; \gamma^1 e^1 = 0, \gamma^2 e^2 = 0 \right\rbrace = \ker \gamma^1 \times \ker \gamma^2 =: \mathcal{X}^1_1 \times \mathcal{X}^2_1$. By assumption $(A2)$, $\mathcal{X}_1$ is then dense in $\mathcal{X}$, and point~$(ii)$ of Definition~\ref{def:BCS} is satisfied.

\noindent\textbf{Point~$(iii)$:} {By assumptions $(A1)$, $(A2)$, and $(A3)$, Theorem~\ref{th:skew-adjoint} applies and $A$ is skew-adjoint on $\mathcal{X}$, so in particular} $(\beta I - A)$ is onto for all $\beta \in \bbC, \Re{\rm e} \beta \neq 0$: the point~$(iii)$ of Definition~\ref{def:BCS} holds.

\noindent\textbf{Point~$(iv)$:} Let $J := \begin{bmatrix} 0 & -K \\ L & 0 \end{bmatrix}$ and $e \in \ker (I - J) \cap \mathcal{X}_1$. Then:
$$
e = A e \in \mathcal{X}_1.
$$
Applying $A^\star = -A$, by Theorem~\ref{th:skew-adjoint}, one gets:
$$
-A e = A^\star A e \in \mathcal{X},
$$
from which it is deduced that:
$$
e = -A^\star A e \in \mathcal{X}.
$$
Multiplying both side by $e$ in $\mathcal{X}$, we obtain $\left\| e \right\|_{\mathcal{X}_1}^2 = 0$. Then $\ker (I - J) \cap \mathcal{X}_1 = \left\lbrace \begin{pmatrix} 0 \\ 0 \end{pmatrix} \right\rbrace$ and point~$(iv)$ of Definition~\ref{def:BCS} holds.

This shows that $\left( J, G \right)$ is indeed a boundary control system on $(\mathcal{Z}^1\times\mathcal{Z}^2, \mathcal{X}^1\times\mathcal{X}^2, \mathcal{U}^1\times\mathcal{U}^2)$. As a first consequence, the control operator $B$ is uniquely determined, as claimed in Proposition~\ref{prop:BCS}, point~4.

\noindent\textbf{Stokes-Dirac structure:} {Starting from~\eqref{eq:formal-adjoint} with the definition of $C := \begin{bmatrix} 0 & \beta^2 \\ \beta^1 & 0 \end{bmatrix}$, one has for all $z := \begin{pmatrix} z^1 \\ z^2 \end{pmatrix} \in \mathcal{Z}$ and all $x := \begin{pmatrix} x^1 \\ x^2 \end{pmatrix} \in \mathcal{Z}$:
$$
\begin{array}{rcl}
\left( J z, x \right)_{\mathcal{X}} + \left( z, J x \right)_{\mathcal{X}}
&=& \left( -K z^2, x^1 \right)_{\mathcal{X}^1} + \left( L z^1, x^2 \right)_{\mathcal{X}^2}
+ \left( z^1, -K x^2 \right)_{\mathcal{X}^1} + \left( z^2, L x^1 \right)_{\mathcal{X}^2}, \\
&=& \left\langle \gamma^1 z^1, \beta^2 x^2 \right\rangle_{\mathcal{U}^1,(\mathcal{U}^1)'} 
+ \left\langle \beta^1 z^1, \gamma^2 x^2 \right\rangle_{(\mathcal{U}^2)',\mathcal{U}^2} \\
&& \qquad \qquad \qquad \qquad + \left\langle \beta^2 z^2, \gamma^1 x^1 \right\rangle_{(\mathcal{U}^1)',\mathcal{U}^1}
+ \left\langle \gamma^2 z^2, \beta^1 x^1 \right\rangle_{\mathcal{U}^2,(\mathcal{U}^2)'} , \\
&=& \left\langle G z, C x \right\rangle_{\mathcal{U},\mathcal{U}'}
+ \left\langle C z, G x \right\rangle_{\mathcal{U}',\mathcal{U}}.
\end{array}
$$
}

From Proposition~\ref{prop:BCS}, point~4, $J = \begin{bmatrix} A|_{\mathcal{X}} & B \end{bmatrix} \begin{bmatrix} I_{\mathcal{Z}} \\ G \end{bmatrix}$, and thus, with the definitions of $\mathcal{F}$ and $\mathcal{E}$, one has for all $\begin{pmatrix}
z \\ u 
\end{pmatrix} \in \mathcal{E}$ {and all $\begin{pmatrix}
x \\ v 
\end{pmatrix} \in \mathcal{E}$}:
$$
\begin{array}{rcl}
\left\langle \begin{bmatrix} A|_{\mathcal{X}} & B \\ -C & 0 \end{bmatrix} \begin{pmatrix}
z \\ u 
\end{pmatrix},
{\begin{pmatrix}
x \\ v 
\end{pmatrix}} \right\rangle_{\mathcal{F},\mathcal{E}}
{+
\left\langle \begin{pmatrix}
z \\ u 
\end{pmatrix},
\begin{bmatrix} A|_{\mathcal{X}} & B \\ -C & 0 \end{bmatrix} \begin{pmatrix}
x \\ v
\end{pmatrix} \right\rangle_{\mathcal{F},\mathcal{E}}
}
&=& \left( J z, {x} \right)_{\mathcal{X}} {+ \left( z, J x \right)_{\mathcal{X}}} - \left\langle C z, {v} \right\rangle_{\mathcal{U}',\mathcal{U}} {- \left\langle u, C x \right\rangle_{\mathcal{U},\mathcal{U}'}}, \\
&=& {\left\langle G z, C x \right\rangle_{\mathcal{U},\mathcal{U}'}
+ \left\langle C z, G x \right\rangle_{\mathcal{U}',\mathcal{U}}} \\
&& \qquad \qquad \qquad
{- \left\langle C z, G x \right\rangle_{\mathcal{U}',\mathcal{U}}
- \left\langle G z, C x \right\rangle_{\mathcal{U},\mathcal{U}'}}, \\
&=& {0}.
\end{array}
$$
This yields that $\mathcal{J} {:=} \begin{bmatrix} A|_{\mathcal{X}} & B \\ -C & 0 \end{bmatrix} \in \mathcal{L}(\mathcal{E},\mathcal{F})$ indeed satisfies~\eqref{eq:skew-sym}. % since:
%Finally, a
Applying Theorem~\ref{th:Stokes-Dirac} shows that the graph of $\mathcal{J}$ {defined as above} is a Stokes-Dirac structure on $\mathcal{B} = \mathcal{F} \times \mathcal{E}$.

\noindent\textbf{Form of $\mathcal{J}$:} Now, it remains to prove that $\mathcal{J} = \begin{bmatrix} A|_{\mathcal{X}} & B \\ -C & 0 \end{bmatrix}$ writes as in~\eqref{eq:extendedInterconnectionOperator}, by showing that indeed:
$$
B = \begin{bmatrix} 0 & B^2 \\ B^1 & 0 \end{bmatrix},
$$
%\quad \text{ and } \quad
%C = \begin{bmatrix} 0 & {\beta^2} \\ {\beta^1} & 0 \end{bmatrix},
%$$
with
$B^1 \in \mathcal{L} (\mathcal{U}^1, \mathcal{X}_{-1}^2)$, 
$B^2 \in \mathcal{L} (\mathcal{U}^2, \mathcal{X}_{-1}^1)$,
%$C^1 \in \mathcal{L} (\mathcal{Z}^1, (\mathcal{U}^2)')$,
%$C^2 \in \mathcal{L} (\mathcal{Z}^2, (\mathcal{U}^1)')$, 
where we recall that $\mathcal{X}_{-1}^i$ is the projection of $\mathcal{X}_{-1}$ on the $i$-th component, for $i=1, 2$.

The form of $B$ entirely relies on its construction, given in the proof of \cite[Proposition~10.1.2]{Tucsnak2009}: $B = (J - A) H$ where $H \in \mathcal{L}(\mathcal{U}, \mathcal{Z})$ is a bounded right inverse of $G$ (which exists since $G$ is onto).

Since $G = \begin{bmatrix} \gamma^1 & 0 \\ 0 & \gamma^2 \end{bmatrix}$, $H = \begin{bmatrix} H^1 & 0 \\ 0 & H^2 \end{bmatrix}$, where $H^i \in \mathcal{L}(\mathcal{U}^i,\mathcal{Z}^i)$ is a bounded right inverse of $\gamma^i$, for $i=1, 2$. By construction with the operators $K$ and $L$ {and the assumption of density of $\mathcal{X}_1$ in $\mathcal{X}$}, $J-A|_{\mathcal{X}} = BG$ is of the form $\begin{bmatrix} 0 & S^2 \\ S^1 & 0 \end{bmatrix}$, which yields  $B = \begin{bmatrix} 0 & S^2 H^2 \\ S^1 H^1 & 0 \end{bmatrix} \in \mathcal{L}(\mathcal{U}^1\times\mathcal{U}^2, \mathcal{X}_{-1}^1\times\mathcal{X}_{-1}^2)$. Hence $B^1 = S^1 H^1$ is related to $\gamma^1$, and $B^2 = S^2 H^2$ is related to $\gamma^2$. {This concludes the proof.}

\end{document}